\documentclass[pdflatex,12pt,a4paper,twoside]{amsart}
\usepackage[a4paper,left=30mm,right=30mm,top=24mm,bottom=28mm]{geometry}

\usepackage{graphicx} 
\usepackage{amsmath, mathtools, amssymb, amsthm}
\usepackage{bbm}
\usepackage{xcolor}
\usepackage{subcaption}

\newtheorem{defi}{Definition}[section]
\newtheorem{thm}[defi]{Theorem}

\newtheorem{prop}[defi]{Proposition}
\newtheorem{lemma}[defi]{Lemma}

\newtheorem{ass}[defi]{Assumption}

\theoremstyle{definition}
\newtheorem{rem}[defi]{Remark}

\numberwithin{equation}{section}

\renewcommand{\div}{\operatorname{div}}
\newcommand{\one}[1]{{\bf{1}}_{#1}}
\newcommand{\1}{{\mathbbm{1}}}

\newcommand{\Lin}{\mathcal{L}}
\newcommand{\CVH}{C_{V \hookrightarrow H}}

\newcommand{\init}{\mathrm{in}}

\newcommand{\N}{\mathbb{N}}
\newcommand{\Z}{\mathbb{Z}}

\newcommand{\R}{\mathbb{R}}

\newcommand{\Oe}{\Omega_\e}
\newcommand{\Oes}{\Omega_\e^\mathrm{s}}
\newcommand{\Ge}{{\Gamma_\e}}
\newcommand{\Le}{\Lambda_\e}

\newcommand{\Yp}{{Y^*}}
\newcommand{\Ypp}{{Y^*_\#}}
\newcommand{\Yptx}{Y^*(t,x)}
\newcommand{\Ys}{{Y^\mathrm{s}}}
\newcommand{\Ysp}{{Y^\mathrm{s}_\#}}
\newcommand{\Ystx}{Y^\mathrm{s}(t,x)}
\newcommand{\Gtx}{\Gamma\tx}

\newcommand{\dd}{\, \mathrm{d}}
\newcommand{\dt}{\dd t}
\newcommand{\ds}{\dd s}
\newcommand{\dx}{\dd x}
\newcommand{\dy}{\dd y}
\newcommand{\dyx}{\dy\dx}
\newcommand{\dxt}{\dx\dt}
\newcommand{\dyxt}{\dy \dx\dt}

\newcommand{\tx}{(t,x)}
\newcommand{\xy}{(x,y)}
\newcommand{\sx}{(s,x)}
\newcommand{\txy}{(t,x,y)}
\newcommand{\stx}{(s,t,x)}
\newcommand{\stxy}{(s,t,x,y)}

\newcommand{\e}{\varepsilon}
\newcommand{\ue}{u_\e}
\newcommand{\ve}{v_\e}
\newcommand{\pe}{p_\e}
\newcommand{\we}{w_\e}
\newcommand{\qe}{q_\e}

\newcommand{\un}{u_0}
\newcommand{\vn}{v_0}

\newcommand{\qO}{q_1}

\newcommand{\hue}{\hat{u}_\e}
\newcommand{\hve}{\hat{v}_\e}
\newcommand{\hpe}{\hat{p}_\e}
\newcommand{\hwe}{\hat{w}_\e}
\newcommand{\hqe}{\hat{q}_\e}
\newcommand{\hQe}{\hat{Q}_\e}
\newcommand{\Qe}{Q_\e}

\newcommand{\fe}{f_\e}
\newcommand{\pbe}{p_{\mathrm{b},\e}}
\newcommand{\vGe}{v_{\Gamma_\e}}
\newcommand{\vein}{v_\e^\init}
\newcommand{\wein}{w_\e^\init}

\newcommand{\hfe}{\hat{f}_\e}
\newcommand{\hpbe}{\hat{p}_{\mathrm{b},\e}}
\newcommand{\hvGe}{\hat{v}_{\Gamma_\e}}
\newcommand{\hvein}{\hat{v}_\e^\init}
\newcommand{\hwein}{\hat{w}_\e^\init}

\newcommand{\hvGn}{\hat{v}_{\Gamma}}

\newcommand{\hun}{\hat{u}_0}
\newcommand{\hvn}{\hat{v}_0}
\newcommand{\hv}{\hat{v}}
\newcommand{\hq}{\hat{q}}
\newcommand{\hqO}{\hat{q}_1}
\newcommand{\hpO}{\hat{p}_1}
\newcommand{\hQ}{\hat{Q}}
\newcommand{\hqb}{\hat{q}_b}

\newcommand{\vnin}{v_0^\init}
\newcommand{\vGn}{v_{\Gamma}}

\newcommand{\pb}{p_b}
\newcommand{\pbO}{p_{b,1}}
\newcommand{\hpb}{\hat{p}_b}
\newcommand{\hpbO}{\hat{p}_{b,1}}
\newcommand{\hp}{\hat{p}}

\newcommand{\psie}{\psi_\e}
\newcommand{\psiem}{\psi_\e^{-1}}
\newcommand{\Pe}{\Psi_\e}
\newcommand{\Ae}{A_\e}
\newcommand{\Je}{J_\e}
\newcommand{\Pem}{\Psi_\e^{-1}}
\newcommand{\PemT}{\Psi_\e^{-\top}}
\newcommand{\PeT}{\Psi_\e^{\top}}
\newcommand{\Jem}{J_\e^{-1}}
\newcommand{\AeT}{A_\e^{\top}}
\newcommand{\Aem}{A_\e^{-1}}
\newcommand{\AemT}{A_\e^{-\top}}

\newcommand{\psin}{\psi_0}
\newcommand{\psinm}{\psi_0^{-1}}
\newcommand{\Pn}{\Psi_0}
\newcommand{\An}{A_0}
\newcommand{\Jn}{J_0}
\newcommand{\Pnm}{\Psi_0^{-1}}
\newcommand{\PnmT}{\Psi_0^{-\top}}
\newcommand{\PnT}{\Psi_0^{\top}}

\newcommand{\Anm}{A_0^{-1}}
\newcommand{\Jnm}{J_0^{-1}}

\newcommand{\intT}{\int\limits_{(0,T)}}
\newcommand{\intO}{\int\limits_{\Omega}}
\newcommand{\intOe}{\int\limits_{\Oe}}

\newcommand{\intY}{\int\limits_{Y}}
\newcommand{\intYp}{\int\limits_{\Yp}}
\newcommand{\intYptx}{\int\limits_{\Yptx}}
\newcommand{\intOYp}{\intO \intYp}
\newcommand{\intTOY}{\intT \intO \intY}
\newcommand{\intTOYp}{\intT \intO \intYp}
\newcommand{\intTOYptx}{\intT \intO \intYptx}
\newcommand{\intTOe}{\intT \intOe}
\newcommand{\intTO}{\intT \intO}

\newcommand{\hzi}{\hat{\zeta_i}}

\newcommand{\Ke}{K_\e}

\newcommand{\tss}[1]{\xrightarrow[]{\makebox[0.5cm]{$#1$}}	\hspace{-0.7cm}\xrightarrow[]{\makebox[0.6cm]{}}}

\newcommand{\tsw}[1]{{\xrightharpoonup[]{\makebox[0.5cm]{$#1$}}\hspace{-0.7cm}\xrightharpoonup[]{\makebox[0.6cm]{}}}}\usepackage{enumitem}
\usepackage{crossreftools}

\title[Darcy law with memory for evolving microstructure]{A Darcy law with memory by homogenisation for evolving microstructure}
\date{February 2024}
\author{David Wiedemann}
\address{Department of Mathematics\\
Technical University of Dortmund\\
Germany}
\email{david.wiedemann@math.tu-dortmund.de}
\address{(D.W.) 
	Institute of Mathematics, University of Augsburg,
	86135 Augsburg, Germany.
	Current address: Department of Mathematics\\
	Technical University of Dortmund, Vogelpothsweg 87, 44227 Dortmund, Germany.}
	\email{david.wiedemann@math.tu-dortmund.de}
\thanks{The research of D.W. was partially supported by a doctoral scholarship provided by the
	Studienstiftung des deutschen Volkes}
\thanks{We would like to thank Dr.~Christoph Zimmer for helpful discussions and comments on differential--algebraic equations.}
\author{Malte A. Peter}
\address{(M.A.P.) Institute of Mathematics, University of Augsburg and
	Centre for Advanced Analytics and Predictive Sciences, University of Augsburg,
	86135 Augsburg, Germany.}
\email{malte.peter@math.uni-augsburg.de}
\thanks{}


\begin{document}
\bibliographystyle{amsplain}
\begin{abstract}
We consider the homogenisation of the instationary Stokes equations in a porous medium with an a-priori given evolving microstructure.
In order to pass to the homogenisation limit, we transform the Stokes equations to a domain with a fixed periodic microstructure.
The homogenisation result is a Darcy-type equation with memory term and has the form of an integro-differential equation. The evolving microstructure leads to a time and space dependent permeability coefficient and the local change of the porosity causes an additional source term for the pressure.
\end{abstract}
\maketitle
\section{Introduction}
Understanding the behaviour of fluid flow in complex porous media or heterogeneous materials is crucial in various scientific and engineering disciplines such as materials science, chemical engineering and geophysics. In many practical scenarios, the porous medium exhibits a heterogeneous microstructure that evolves over time due to processes such as phase transitions, chemical reactions or mechanical deformation. The prediction of flow properties in such evolving microstructures poses significant challenges, necessitating advanced mathematical models.

The Stokes equations govern the motion of a viscous fluid. They have been extensively studied in the context of flow through porous media by means of homogenisation. So far most of the homogenisation results are derived for fixed microstructure. 

\subsection*{Goal of this work}
In this work, we consider the homogenisation of the instationary Stokes equations in a porous medium with evolving microstructure at small Reynolds number.
We consider a time interval $(0,T)$ for $T >0$. Let $d \in \N$, we denote the $\e$-scaled pore space at time $t\in (0,T)$ by $\Oe(t) \subset \R^d$. We denote the interface of the pore space with the solid matrix domain at $t \in (0,T)$ by $\Ge(t)$ and the boundary of the pore space at the outer boundary at $t \in (0,T)$ by $\Lambda_\e(t)$.
Such evolving geometry is illustrated in Figure \ref{fig:GeometryTwoDifferentTimePoints}.
\begin{figure}[h]	
	\centering
	\includegraphics[width=0.8\linewidth]{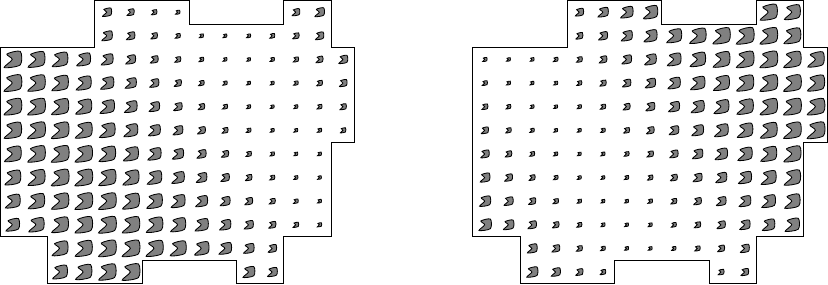} 
	\caption{Illustration of amicroscopically evolving geometry at two different points in time}
	\label{fig:GeometryTwoDifferentTimePoints}
\end{figure}

At the interface $\Ge(t)$, we assume a non-homogeneous Dirichlet boundary condition with given boundary values $\vGe$, which can model a no-slip boundary condition for the evolving domain.
At the outer boundary $\Le(t)$ of the porous medium, we assume a normal stress boundary condition with normal stress $\pbe$, which models fluid in- and outflow.

Let $\mu>0$ be the fluid viscosity, $\fe$ the density of the bulk force and $\nu$ the unit outer normal vector of $\Oe$. 
We consider the fluid velocity $\ve$ and the pressure $\pe$ as the solution of:
\begin{subequations}\label{eq:Strong:Stokes:Eps}
\begin{align}\label{eq:Strong:Stokes:Eps:1}
\partial_t \ve - \mu \e^2 \div( \nabla \ve + (\nabla \ve)^\top) + \nabla \pe &= \fe && \textrm{in } \Oe(t), t\in (0,T)\, ,
\\\label{eq:Strong:Stokes:Eps:2}
\div \ve &= 0 && \textrm{in } \Oe(t), t\in (0,T) \, ,
\\\label{eq:Strong:Stokes:Eps:3}
\ve &= \vGe && \textrm{on } \Ge(t), t\in (0,T) \, ,
\\\label{eq:Strong:Stokes:Eps:4}
\left( -\mu \e^2 \div( \nabla \ve + (\nabla \ve)^\top) + \pe \1 \right) \nu &= \pbe && \textrm{on } \Le(t), t\in (0,T) \, ,
\\\label{eq:Strong:Stokes:Eps:5}
\ve(0) &= \vein && \textrm{in } \Oe(0) \, .
\end{align}
\end{subequations}

We show that the extension of the fluid velocity $\ve$ by zero and some extension of the pressure $\pe$ converges weakly as $\e \to 0$ to the solution $(v,p)$ of the Darcy-type law with memory \eqref{eq:Strong:DarcyMemory}.
The effective equations are defined on the macroscopic limit domain $\Omega$. The domain $\Omega$ is approximated by $\Oe(t)$ in the sense that $\Omega$ is the interior of the support of the weak limit of the characteristic function of $\overline{\Oe(t)}$, i.e.~$\Omega = \operatorname{int}\left(\operatorname{supp}\left( \one{\overline{\Oe(t)}} \right)\right)$, where $\one{U}$ denotes the characteristic function for a measurable set $U$. A precise definition of $\Omega$ is given below.

\subsection*{Homogenisation result}
The homogenisation result is a Darcy-type law with memory and is given by the following integro--differential equation:
\begin{subequations}\label{eq:Strong:DarcyMemory}
\begin{align}\label{eq:Strong:DarcyMemory:1}
v\tx &=a^\init\tx + \frac{1}{\mu}\int\limits_0^t K\stx (f\sx -\nabla p\sx) \ds && \textrm{in } (0,T) \times \Omega \,,
\\\label{eq:Strong:DarcyMemory:2}
\div (v) &= -\frac{\dd}{\dt} \Theta && \textrm{in } (0,T) \times \Omega\, ,
\\\label{eq:Strong:DarcyMemory:3}
p &= \pb && \textrm{on } (0,T) \times \partial \Omega\, ,
\end{align}
\end{subequations}
The permeability-type coefficient $K$ and the initial velocity $a^\init$ can be computed by means of the solutions of the cell problems \eqref{eq:Strong:CellProblem:d} and \eqref{eq:Strong:CellProblem:init}.
In \eqref{eq:Strong:DarcyMemory:2}, the right-hand side of the divergence condition is formulated for the case of a no-slip boundary condition at the fluid--solid interface in the microscopic model, i.e.~for the case that the velocity $\vGe$ and thus the fluid velocity at the interface $\Oe(t)$ is equal the velocity of the interface. For this model, we can simplify the right-hand side of \eqref{eq:Strong:DarcyMemory:2} to $ -\frac{\dd}{\dt} \Theta$, where $\Theta$ is the porosity of the local reference cell $\Yptx$ at the macroscopic porsition $x \in \Omega$ at time $t\in (0,T)$. For a general velocity field $\hvGe$, the right-hand side depends on its two-scale limit and is formulated in \eqref{eq:Strong:DarcyMemory:GenDirich}.

The cell problems are defined on the local evolving reference cells $\Yptx$, where $\Yptx\subset (0,1)^d$ is given by the two-scale limit of $\Oe(t)$ in the sense that $\one{\Oe(t)}(x)$ two-scale converges to $\one{\Yptx}(y)$, where the periodic extension of $\Yptx$ is for a.e.~$\tx \in \Omega \times (0,T)$ a Lipschitz set.

The permeability tensor $A\stx$ is defined for a.e.~$\tx \in (0,T) \times \Omega$ and every $s \in(0,t)$ and $i,j \in\{1,\dots, d\}$ by
\begin{align}\label{eq:def:K}
K_{ji} \stx \coloneqq \int\limits_{\Yptx} \zeta_i\stxy\cdot e_j \dy \,,
\end{align}
where $(\zeta_i(s,x,t,y),\pi_i(s,x,t,y))$ for $i \in\{1, \dots, d\}$ are the solutions of the cell problems \eqref{eq:Strong:CellProblem:d}. The parameters $(s,x) \in (0,T) \times \Omega$ denote the initial time for the cell problem and the macroscopic position, respectively. 
\begin{subequations}\label{eq:Strong:CellProblem:d}
\begin{align}
	\partial_t \zeta_i - \Delta_{yy} \zeta_i + \nabla_y \pi_i &= 0 && \textrm{in } \Yptx, \, t \in (s,T) \, ,
	\\
	\div_y \zeta_i &= 0 && \textrm{in } \Yptx, \, t \in (s,T) \, ,
 \\
 \zeta_i &= 0 && \textrm{on } \Gtx, \, t \in (s,T) \, ,
	\\
	\zeta_i &= e_i && \textrm{in } Y^*(s,x)\, .
\end{align}
\end{subequations}
The initial value $a^\init$ is given by
\begin{align}\label{eq:initialValue:Darcy}
a^\init\tx \coloneqq \int\limits_{\Yptx} \zeta_0\tx \dy,
\end{align}
where $(\zeta_0(x,t,y),\pi_0(x,t,y))$ is the solution of the following cell problem \eqref{eq:Strong:CellProblem:init}:
\begin{subequations}\label{eq:Strong:CellProblem:init}
\begin{align}
	\partial_t \zeta_0 - \mu\Delta_{yy} \zeta_0 + \nabla_y \pi_0 &= 0 && \textrm{in } \Yptx, \, t \in (0,T) \, ,
	\\
	\div_y \zeta_0 &= 0 && \textrm{in } \Yptx, \, t \in (0,T) \, ,
    \\
    \zeta_0 &= 0 && \textrm{on } \Gtx, \, t \in (0,T) \, ,
	\\
	\zeta_0 &= v_0^\init && \textrm{in } Y^*(0,x) \, 
\end{align}
\end{subequations}
and $v_0^\init$ is the two-scale limit of the initial values $\vein$ of the Stokes problem.

\subsection*{Homogenisation approach}
In order to homogenise \eqref{eq:Strong:Stokes:Eps}, we transform the evolving domain on a periodically perforated fixed reference domain.
We homogenise the resulting substitute equations on this substitute domain. This leads to two-pressure Stokes equations in the (time cylindrical) two-scale substitute domain. Separating the microscopic and macroscopic spatial variable leads to a Darcy law with memory for evolving microstructure complemented by cell problems. 
We transform the two-pressure Stokes equations, the Darcy law with memory for evolving microstructure and the associated cell problems back to the evolving local reference cell. This leads in particular to the transformation-independent homogenisation result \eqref{eq:Strong:DarcyMemory}.
This approach is illustrated in \ref{fig:TwoScaleTrafoMethod}.
\begin{figure}[h]	
	\centering
	\includegraphics[width=1.\linewidth]{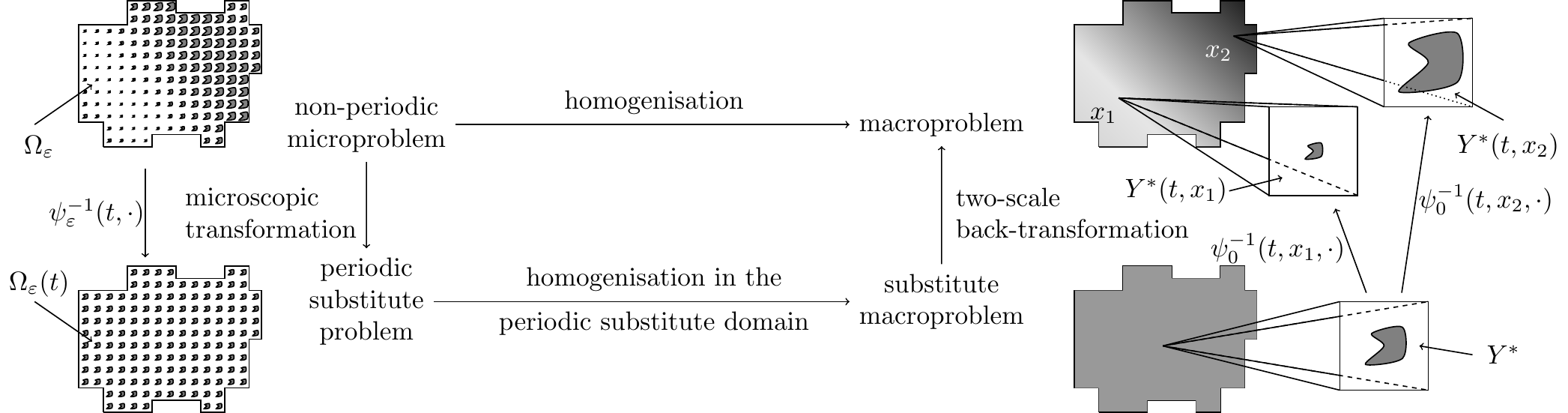} 
	\caption{Two-scale transformation method}
	\label{fig:TwoScaleTrafoMethod}
\end{figure}

In \cite{AA23} it is shown that the transformation and homogenisation commute and, thus, the Darcy law with memory for evolving microstructure \eqref{eq:Strong:DarcyMemory} is also the limit result for the Stokes equations \eqref{eq:Strong:Stokes:Eps}.

\subsection*{Literature overview}
Based on the results of experiments, Darcy presented a fundamental principle of fluid mechanics in porous media \cite{Dar56}. Darcy's law states that the rate of flow through porous media is directly proportional to the negative hydraulic gradient and inversely proportional to the viscosity of the fluid with the permeability coefficient as proportionality factor.
It can be derived mathematically by means of homogenising the (Navier--)Stokes equations in a perforated domain. In particular, this mathematical approach provides a better understanding of the effects of the microscopic geometry on the permeability coefficient. First upscaling approaches used formal two-scale asymptotic expansion and are presented in \cite{Kel80, Lio81, San80}.

The main difficulty in the rigorous homogenisation of the Stokes equations lies in the uniform a-priori estimate of the pressure. Tartar overcame this problem by constructing a restriction operator \cite{Tar80} and provided a rigorous proof of the homogenisation procedure. This operator was extended by Allaire to allow the homogenisation in the case where the solid space of the porous medium is also connected \cite{All89}. A modification of this restriction operator \cite{LA90} allowed the consideration of different boundary conditions at the pore interfaces. Furthermore, an extension of the restriction operator from $H^1$ to $W^{1,p}$ integrability enables the homogenisation of the Navier--Stokes equations \cite{Mik91}.
A different approach for the derivation of the a-priori estimates was presented by Zhikov in \cite{Zhi94}, who constructed a family of $\e$-scaled operators, which are right-inverses of the divergence operator. In particular, these operators enable a construction of a restriction operator in the sense of \cite{Tar80} with weaker estimates, which are still sufficient in order to show the strong convergence of the pressure \cite{Mik98}. For the construction of these right-inverse divergence operators, the extension operators of \cite{ACM+92} are used. In particular, such $\e$-scaled right-inverse operators become useful for the homogenisation of the compressible (Navier--)Stokes equations \cite{Mas02} or in our case, where the domain evolution motivates inhomogeneous Dirichlet boundary conditions leading to an inhomogeneous divergence condition.
While these works considered Dirichlet or periodic boundary conditions at the boundary of the macroscopic domain, the case of normal stress boundary conditions is considered in \cite{FMW17}.

The upscaling of the instationary Stokes equations was first studied by formal two-scale asymptotic expansion in \cite{Lio81} and rigorous homogenisation results are proven in \cite{All92a} and \cite{Mik94}. The result is a Darcy law with memory, which is an integro-differential equation and can be approximated for large times and constant force by the classical Darcy law \cite{Mik94}.
However, the $\e$-scaling of the viscosity becomes crucial and, for different scaling, the time derivative can vanish in the homogenisation limit leading directly to the stationary Darcy equation \cite{Mik91}.

The above-mentioned works considered the case where the porosity remains constant for $\e \to 0$. For the case of isolated obstacles it is possible to scale the obstacles asymptotically smaller than the periodicity size $\e$, i.e.~the obstacles are of size $\e^{\alpha}$ for $\alpha>1$ \cite{All91a, All91b}. The homogenisation result depends on the exact value of $\alpha$ and leads for asymptotically small obstacles to the Stokes equations itself, for critically scaled obstacles to the Brinkman equation and for asymptotically large obstacles to a Darcy law. The permeability tensor for the Darcy law differs from the case of obstacles of size $\e$, see \cite{All91d}.

The above mentioned homogenisation results deal with the case of a fixed microstructure. For an evolving domain, the quasi-stationary Stokes equations are homogenised in \cite{JDE24}. There the geometrical setting is the same as in the work presented here, but the Stokes equations are considered without the time-derivative term.

The consideration of an evolving microstructure is motivated by many different physical, chemical and biological applications. For example, for dissolution and precipitation in a porous medium, a precipitate layer may be added to or be dissolved from the pore walls, implying that the overall solid part (and, implicitly, the void space) is evolving.
In \cite{Noo08, RNF12, REK15, BBPF16, RRP16, SK17, SK17a, BVP20}, such processes are modelled as free boundaries by means of a level-set function or phase-field approaches.
However, these models are only formally upscaled by asymptotic two-scale expansions.
For fixed microstructure, related (advection–) diffusion problems are homogenised rigorously in \cite{CDT03, MP05, MZ11}.

For given evolving microstructure evolution, rigorous homogenisation results are presented in \cite{Pet07, Pet07a, Pet09, EM17, GNP21, JDE24}. There, the equations are transformed to a fixed microstructure and the resulting substitute problems are homogenised. 
For a general class of transformations, it was shown in \cite{AA23} that the homogenisation and the transformation commutes, which justifies this transformation approach. Moreover, it was shown how the two-scale limit equations and the cell problems can be transformed back into a transformation-independent limit result.
In \cite{JDE24}, the quasi-stationary Stokes equations for evolving microstructure are homogenised. There, as also in this work, the transformation to the fixed periodically perforated substitute domain leads to transformation matrices in the symmetrised gradient in the substitute equations. A uniform Korn-type inequality for those two-scale transformed symmetric gradients is derived in \cite{JDE24}. We use this Korn-type inequality  also for the derivation of the a-priori estimates here.
This two-scale transformation approach was also used in \cite{GP23, NA23} for the rigorous homogenisation of an reaction--diffusion problem with free boundary where the evolution of the domain is coupled with the unknown concentration.

The homogenisation of fluid flow in evolving porous media is also important for problems in poroelasticity. The first linear theory was developed by Biot (cf.~\cite{B55,B56}). Starting with a description of the microporomechanics by equations of elasticity and fluid flow, effective equations can also be derived by means of homogenisation (cf.~\cite{L79}, \cite{BK81}). However, in order to homogenise rigorously, the Stokes problem was linearised by assuming that the fluid domain is constant in time (cf.~\cite{Mei14}). 
Recently, the corresponding non-linear model received considerable attention (cf.~\cite{BPE14, CBH17, MP21}). However, these works passed to the homogenisation limit only formally. In this paper, we provide a rigorous homogenisation result for the decoupled Stokes-problem, which is a step towards the homogenisation of the fully coupled fluid--structure interaction problem.

\subsection*{Organisation of this paper}
This paper is organised as follows:
In Section \ref{sec:Eps-scaled-Problem}, we formulate the $\e$-scaled problem, the instationary Stokes equations in the evolving domain. We present the assumptions on the domain and its evolution by means of the periodically perforated reference domain and the existence of transformation mappings.
In Section \ref{sec:Trafo-eps-scaled}, we transform the Stokes equations to the substitute domain.
For the resulting substitute problem, we show the existence and uniqueness of a solution as well as uniform a-priori estimates in Section \ref{sec:Existence:eps-scaled}.
Having the a-priori estimates, we can pass to the homogenisation limit $\e \to 0$ for the substitute equations in Section~\ref{sec:Limit-process}. This leads to a system of two-pressure Stokes equations in the in time cylindrical two-scale domain.
In Section \ref{sec:DarcySubst}, we separate the micro- and macroscopic spatial variable in the limit equations and derive a Darcy law with memory with cell problems defined on the fixed substitute cell but with transformation coefficient.
We transform the two-pressure Stokes equations and the Darcy law with it cell problems back to the actual evolving domain in Section \ref{sec:BackTrafo}. The result is the Darcy law with memory for evolving microstructure \eqref{eq:Strong:DarcyMemory}. This homogenised equation as well as the cell problems are formulated without transformation quantities in the evolving domain and, hence, are transformation-independent. 

\subsection*{Notations}
Let $d, n,m \in \N$ and $U\subset \R^d$. For a function $u \colon U \to \R$, a vector field $v \colon U\subset \R^{n}$ and a matrix-valued function $M \colon \R^{m \times n}$, we use the following notation for its derivatives. For $x \in U$, we write $\nabla u(x) \in \R^d$ for the gradient of $u$ at $x\in U$, i.e.~$(\nabla u)_i(x) \coloneqq \partial_{x_i} u(x)$, and $\partial_x u(x) \coloneqq \nabla u^\top(x)\in \R^{1 \times d}$ for its transposed. We denote the Jacobian matrix of $v$ at $x \in U$ by $\nabla^\top v(x) \coloneqq \partial_x v(x) \in \R^{n \times d}$ i.e.~$\partial_x v(x)_{ij} \coloneqq \partial_{x_j} v_i(x)$ and its transposed by $\nabla v(x) = \partial_x v^\top(x)$. 
Moreover, for $v \colon U\subset \R^{d}$, we define the divergence $\div v(x) = \sum_{i =1}^d\partial_{x_i} v_i(x)$.
For a matrix-valued function $M$, we write $\partial_x M(x) \in \R^{(m \times n) \times d}$ for its derivative at $x \in U$, i.e.~$\partial_x A(x)_{jki} \coloneqq \partial_{x_i} A_{jk}(x)$ and $\nabla A(x) \coloneqq (\partial_x A(x))^\top \in \R^{d \times (m \times n)}$, where the transposed is defined by $\nabla A(x)_{ijk} \coloneqq \partial_x A(x)_{jki} = \partial_{x_i} A_{jk}(x)$.
Moreover, for a matrix-valued function $M \colon U \to \R^{d\times n}$, we define the divergence by its columns, i.e.~$\div(A(x)) \in \R^n$ with $\div(A(x))_j \coloneqq \div((A(x)_{ij})_{i=1}^d)$. 
Having the above notations, we can define the scalar- and vector-valued Laplace operator, i.e.~for $u \colon U \to \R$ and $v \colon U \to \R^n$, we define $\Delta u \coloneqq \div \nabla u(x) = \sum\limits_{i=1}^d\partial_{x_i} \partial_{x_i} u(x)$ and $\Delta v(x) \coloneqq \div \nabla v(x) = (\sum_{i=1}^d\partial_{x_i} \partial_{x_i} v_j(x))_{j=1}^n$ for $x \in U$, which gives $(\Delta v(x))_j = \Delta v_j(x)$.

For these notations, we have the following product rules $\partial_x(uv) = v \partial_x u + u \partial_x v$, $\partial_x (uA) = a \partial_x u + u \partial_x A$, $\partial_x (Av) = v^\top \partial_x A + A \partial_x v$, $\div(uv) = u \div(v) + \nabla u \cdot v$, $\div(uA) = u \div(A) + A^\top : \nabla u$, $\div(Av) = \div(A) \cdot v + A :\nabla v$.

We write $\1$ for the identity matrix and $\operatorname{Adj}(A)$ for the adjugate matrix of $A$, i.e.~$\operatorname{Adj}(A) A = \det(A) \1$.
With the above notation for derivatives, the Piola identity is written as $\div(\operatorname{Adj}(\partial_x v)) = 0$.

We use the index $\#$ to denote the periodicity of a function space, i.e.~for a domain $U \subset (0,1)^d$, $C_\#(U)$ denotes the subset of continuous functions on $\R^n$, which are $Y$-periodic. Similarly, we write $H^1_{\#}(U)$ to indicate the periodicity. Moreover, for a $V \subset \partial U$, we write $C_V(U)$ and $H^1_V(U)$ for the restriction of functions which are zero on $V$ or have zero trace on $V$, respectively. We combine these subscripts in order to indicate the restriction to the intersection of the corresponding subsets, i.e.~$H^1_{\# V}(U) \coloneqq H^1_{\#}(U) \cap H^1_{V}(U)$. We denote by $L^2_0(U)$ the subset of $L^2(U)$ with zero mean.

We use $C>0$ as generic constant which can change during estimates but is independent of $\e$. 

\section{The $\e$-scaled problem}\label{sec:Eps-scaled-Problem}
\subsection*{Geometry}
We describe the evolution of the geometry by means of a family of time-dependent and $\e$-scaled diffeomorphisms $\psie$, which map a periodically perforated reference domain $\Oe$ onto the actual domain $\Oe(t)$ at time $t \in [0,T]$. 
We formulate the assumptions on the domains $\Oe(t)$ indirectly by means of assumptions on the reference domain $\Oe$ and the diffeomorphisms $\psi_\e$.

Let $(\e_n)_{n \in \N}$ be a positive sequence converging to zero, as for instance $\e_n = n^{-1}$. In the following we write $\e = (\e_n)_{n \in \N}$.
\subsubsection{Reference structure}
\paragraph{Macroscopic domain:}
We assume that the macroscopic domain $\Omega \subset \R^d$ is open and bounded and consists of entire $\e$-scaled cells $Y = (0,1)^d$, i.e.~let $\Ke \coloneqq \{ k \in \Z^d \mid \e (k +Y) \subset \Omega \}$, we assume that $\Omega = \operatorname{int} \left(\bigcup\limits_{k \in \Ke} \e( k + \overline{Y}) \right)$.

\paragraph{Reference pore geometry:}
We denote the open reference pore space in the periodicity cell by $\Yp \subset Y$ and its complementary solid part by $\Ys \coloneqq Y \setminus \overline{\Yp}$. We denote the periodic extensions of $\Yp$ and $\Ysp$ by $\Ypp \coloneqq \operatorname{int} \left(\bigcup\limits_{k \in \Z^d} k + \overline{\Yp} \right)$ and $\Ysp \coloneqq \operatorname{int} \left(\bigcup\limits_{k \in \Z^d} k + \overline{\Ys} \right)$, respectively. We denote the interface of the pore and solid domain by $\Gamma\coloneqq \partial \Ypp \cap \partial \Ysp \cap [0,1]^d$.

 We assume that:
\begin{itemize}
	\item $0 < |\Yp|, \, |\Ys| < 1$,
	\item $\Ypp$ and $\Ysp$ are open sets with $C^1$-boundary, which are locally located on one side of their boundary,
	\item $\Ypp$ is connected.
\end{itemize}
For a detailed discussion of these assumption see \cite{All89}.
\paragraph{The $\e$-scaled reference domains:}
The $\e$-scaled reference pore space $\Oe$, the $\e$-scaled reference solid space $\Oes$, their interface $\Ge$ and the reference outer boundary $\Le$ are given by
\begin{align*}
\Oe \coloneqq \Omega \cap \e \Ypp \, , 
\qquad 
\Oes \coloneqq \Omega \cap \e \Ysp \, ,
\qquad 
\Ge \coloneqq \Omega \cap \partial \Oe \, ,
\qquad 
\Le \coloneqq \partial \Omega \cap \partial \Oe \, .
\end{align*}

\subsubsection{Evolving microdomain}
In order to define the domains $\Oe(t)$, we use a family of mappings $\psie : [0,T] \times \Omega \to \Omega$. 
At time $t \in [0,T]$, we define the pore space $\Oe(t)$, the solid space $\Oes(t)$, their interface $\Ge(t)$ and the outer boundary $\Le(t)$ by
\begin{align*}
\Oe(t) \coloneqq \psie(t, \Oe) \, , \quad 
\Oes(t) \coloneqq \psie(t, \Oes) \, , \quad 
\Ge(t) \coloneqq \psie(t, \Ge) \, , \quad 
\Le(t) \coloneqq \psie(t, \Le) \, .
\end{align*}
We define the time--space sets by
\begin{align*}
	&\Oe^T \coloneqq \bigcup\limits_{t \in [0,T]} \{t\} \times \Oe(t) \, , \quad 
	&&\Omega_\e^{\mathrm{s}T} \coloneqq \bigcup\limits_{t \in [0,T]} \{t\} \times \Oes(t) \, , \\ 
	&\Gamma_\e^T \coloneqq \bigcup\limits_{t \in [0,T]} \{t\} \times \Ge(t) \, , \quad 
	&&\Le^T \coloneqq \bigcup\limits_{t \in [0,T]} \{t\} \times \Le(t) \, .
\end{align*}
\paragraph{Assumptions on the transformations:}

\begin{ass}[Assumptions on the transformations]\label{ass:psie}
We assume that $\psie$ has the following regularity:
\begin{enumerate}
\item[{\crtcrossreflabel{(R1)}[item:R1]}] $\psie \in C^{1}([0,T];C^2(\overline{\Omega};\R^d))$\,,
\item[{\crtcrossreflabel{(R2)}[item:R2]}] $\psie(t, \, \cdot \,)$ is a $C^2$-diffeomorphism from $\overline{\Omega}$ onto $\overline{\Omega}$ for every $t \in [0,T]$\,.
\end{enumerate}
We assume that $\psie$ satisfies the following uniform bounds:
\begin{itemize}
\item[{\crtcrossreflabel{(B1)}[item:B1]}] $\e^{l-1} \|\psie -x\|_{C^{1}([0,T]; C^l(\overline{\Omega}))} \leq C$ for $l \in \{0,1,2\}$\,,
\item[{\crtcrossreflabel{(B2)}[item:B2]}] $\det(\partial_x\psie \tx ) \geq c_J$ for all $\tx \in [0,T] \times \Omega$ and some constant $c_J >0$\,. 
\end{itemize}
For the asymptotic behaviour of $\psie$, we assume that there exists a limit function $\psin$, which satisfies the following regularity 
\begin{itemize}
\item[{\crtcrossreflabel{(L1)}[item:L1]}] $\psin \in L^\infty(\Omega; C^1([0,T];C^2(\overline{Y}; \R^d)))$\,,
\item[{\crtcrossreflabel{(L2)}[item:L2]}] $\psin(t,x,\cdot ) : \overline{Y} \to \overline{Y}$ is for every $t \in[0,T]$ and a.e.~$x\in \Omega$ a $C^2$-diffeomorphism,
\item[{\crtcrossreflabel{(L3)}[item:L3]}] the displacement mapping $y \mapsto \psin(t,x,y) -y$ can be extended $Y$-periodically i.e.~$(y \mapsto \psin(t,x,y) -y) \in L^\infty(\Omega; C^1([0,T];C^2_\#(\overline{Y}; \R^d)))$
\end{itemize}
and we assume that the following strong two-scale convergences hold
\begin{itemize}
\item[{\crtcrossreflabel{(A1)}[item:A1]}] $\e^{-1} (\psie\tx-x) \tss{2,2} \psin\txy -y$\,,
\item[{\crtcrossreflabel{(A2)}[item:A2]}] $\partial_x \psie \tss{2,2} \partial_y \psin$\,,
\item[{\crtcrossreflabel{(A3)}[item:A3]}] $\e \partial_x \partial_x \psie \tss{2,2} \partial_y \partial_y \psin$\,
\item[{\crtcrossreflabel{(A4)}[item:A4]}] $\e^{-1} \partial_t \psie \tss{2,2} \partial_t \psin$\,,
\item[{\crtcrossreflabel{(A5)}[item:A5]}] $\partial_x \partial_t \psie \tss{2,2} \partial_y \partial_t \psin$\,,
\item[{\crtcrossreflabel{(A6)}[item:A6]}] $\e\partial_x \partial_x \partial_t \psie \tss{2,2} \partial_y \partial_y \partial_t \psin$\,.
\end{itemize}
\end{ass}

The convergences in Assumption~\ref{ass:psie}\ref{item:A1}--\ref{item:A6} denote the strong two-scale convergence (see Definition \ref{def:two-scale}). Due to the uniform essential boundedness, which is given by Assumption~\ref{ass:psie}\ref{item:B1}, the strong two-scale convergences in Assumption~\ref{ass:psie}\ref{item:A1}--\ref{item:A6} hold also for arbitrary $p \in (1,\infty)$ instead of $2$.

We use the following notation for the transformation quantities:
\begin{align*}
\Pe \coloneqq \partial_x \psie \, \qquad \Je\coloneq \det(\Pe) \, \qquad \Ae \coloneqq \operatorname{Adj}(\Pe) \, ,
\\
\Pn \coloneqq \partial_y \psin \, \qquad \Jn \coloneq \det(\Pn) \, \qquad \An \coloneqq \operatorname{Adj}(\Pn) \,.
\end{align*}
We note that the above assumption ensure that $\Je \geq C_J$ and, thus, $\Pe$ is invertible and it holds $\Ae = \Je \Pem$. The uniform bound of $\Je$ from below can be transferred to $\Jn$ via the strong two-scale convergence of $\partial_x \psie$ and one gets $\Jn \geq cJ$ and, thus, also $\An = \Jn \Pnm$. 

For clarification, we note that the uniform bounds in Assumption \ref{ass:psie}\ref{item:B1}  give
\begin{align*}
\e^{-1} \| \psie -x \|_{L^\infty((0,T) \times \Omega)} + 
\|\partial_x \psie \|_{L^\infty((0,T) \times \Omega)} + 
\e \|\partial_x \partial_x \psie \|_{L^\infty((0,T) \times \Omega)} &\leq C \, ,
\\
\e^{-1} \|\partial_t \psie \|_{L^\infty((0,T) \times \Omega)} + 
\|\partial_x \partial_t\psie \|_{L^\infty((0,T) \times \Omega)} + 
\e \|\partial_x \partial_x \partial_t \psie \|_{L^\infty((0,T) \times \Omega)}& \leq C \,.
\end{align*}
\begin{rem}
The regularity assumptions on $\psie$ allow us to transform the Stokes equations in the reference domain. 
The uniform estimates on $\psie$ and its derivatives are crucial for the derivation of the uniform a-priori estimates. The asymptotic behaviour of $\psie$ ensures that the coefficients in the transformed Stokes equations strongly two-scale converge and, hence, we can pass to the homogenisation limit. Moreover, it guarantees that the homogenisation of the actual problem is equivalent to the homogenisation of the transformed problem (see \cite{AA23}). 
\end{rem}

\paragraph{Two-scale limit domain:}
The two-scale limit domain of $\Oe^T$ should not be understood as a domain in $(0,T) \times \Omega \times Y$ but rather as family of domains $\Yptx \subset Y$ with parameters $\tx \in (0,T) \times \Omega$. In particular, for our homogenisation task it is not even necessary that $\Yptx$ is defined for every $x \in \Omega$. Indeed, it suffices that it is defined for a.e.~$x\in \Omega$, where the null-set has to be chosen independent of the time $t\in [0,T]$. Nevertheless, at some points it simplifies the notation if one defines $\Yptx$ for every $x \in \Omega$ and defines the measurable set $\Omega_0^T$ as:
\begin{align*}
\Omega_0^T \coloneqq \bigcup\limits_{\tx \in (0,T) \times \Omega} \{t\} \times \{x\} \times \Yptx \,.
\end{align*}
The set $\Omega_0^T$ and the domains $\Yptx$ can be obtained by means of the two-scale convergence of the characteristic function of $\Oe^T$.
In this sense, we obtain for the reference domain $\Oe$ as two-scale limit $\Yp$ for every $t \in [0,T]$ and a.e.~$x\in \Omega$.
The two-scale limit of the characteristic function $\one{\Oe}$ is given by the function $\one{\Omega \times \Yp}$, which is an element of $L^p(\Omega \times Y)$ and, thus, it does not define the domain uniquely. 
Indeed, for a.e.~$x\in \Omega$, $\one{\Omega \times \Yp}(x, \cdot)$ provides only the domain $\Yp \setminus N_1(x) \cup N_2(x)$ up to null sets $N_1(x), N_2(x) \subset Y$.
The non-uniqueness can be addressed by requiring that for a.e.~$x \in \Omega$ the periodic extension of the domain is a Lipschitz domain and we get $\Yp = \Yp \setminus N_1(x) \cup N_2(x)$. We
address the non-uniqueness of the two-scale limit representative of $\one{\Oe(t)}$ in the same way. This provides the sets $\Yptx$ for every $t\in [0,T]$ and a.e.~$x\in \Omega$. Lemma \ref{lem:TwoScaleEquiv} shows that
$\one{(0,T) \times \Oe}\tx \tss{2,2} \one{(0,T) \times \Omega \times \Yp}\txy$ if and only if $\one{[0,T] \times \Oe}(t,x,\psiem\tx) \tss{2,2} \one{\Omega \times \Yp}(t,x,\psinm\txy)$. Thus, we can determine the two-scale limit for $\Oe(t)$ and $\Oes(t)$ by
\begin{align*} 
\Yptx &= \psin(t,x,\Yp) &&\textrm{ for every } t\in [0,T] \textrm{ and a.e.~}x\in \Omega \, ,
\\
\Ystx &= \psin(t,x,\Ys) &&\textrm{ for every } t\in [0,T] \textrm{ and a.e.~}x\in \Omega\, .
\end{align*}
Their interface, is given by
\begin{align*}
\Gtx &= \psin(t,x,\Gamma) &&\textrm{ for every } t\in [0,T] \textrm{ and a.e.~}x\in \Omega \,
\end{align*}
and we define analogously to $\Omega_0^T$ the solid region $\Omega_0^{\mathrm{s}T}$ by
\begin{align*}
\Omega_0^{\mathrm{s}T} \coloneqq \bigcup\limits_{\tx \in (0,T) \times \Omega} \{t\} \times \{x\} \times \Ystx \,.
\end{align*}

\subsection*{Weak formulation of the $\e$-scaled problem}
Having defined the geometry, we can define the weak formulation for \eqref{eq:Strong:Stokes:Eps}.

We assume that the Dirichlet boundary values $\vGe$ and the normal pressure at the outer boundary can be extended to $\Oe(t)$. We subtract these extensions from the unknowns $\ve$ and $\pe$ and define
\begin{align*}
	\we \coloneqq \ve - \vGe \, , \qquad \wein \coloneqq \vein- \vGe(0)\, ,\qquad \qe \coloneqq \pe - \pbe \, .
\end{align*}

We use this substitution in \eqref{eq:Strong:Stokes:Eps}, multiply the resulting equation by $\varphi$ and integrate it over $\Oe(t)$ and $(0,T)$. Integrating the divergence terms as well as the term with $\qe$ by parts and using the normal stress boundary condition \eqref{eq:Strong:Stokes:Eps:4} leads to the following weak form \eqref{eq:Weak:Stokes:Eps}:
Find $(\we,\qe) \in L^2(0,T;H^1_{\Ge(t)}(\Oe(t); \R^d)) \times L^2(\Oe^T)$ with $\partial_t \we \in L^2(\Oe^T;\R^d)$ such that
\begin{subequations}\label{eq:Weak:Stokes:Eps}
	\begin{align}
		\begin{aligned}\label{eq:Weak:Stokes:Eps:1}
			\int\limits_0^T &\int\limits_{\Oe(t)} \partial_t \ve \cdot \varphi + \e^2 \mu 2 \nabla^\textrm{s}\we : \nabla \varphi - \qe \div(\varphi) \dxt 
			\\
			&=
			\int\limits_0^T \int\limits_{\Oe(t)} \fe \cdot \varphi -\partial_t \hvGe \cdot \varphi - \e^2 \mu 2 \nabla^\textrm{s}\hvGe : \nabla \varphi - \nabla \pbe \cdot \varphi \dxt \, ,
		\end{aligned}
		\\
		\begin{aligned}\label{eq:Weak:Stokes:Eps:2}
		\int\limits_0^T &\int\limits_{\Oe(t)} \div (\we) \phi \dxt = - \int\limits_0^T \int\limits_{\Oe(t)} \div (\vGe) \phi \dxt \, ,
		\end{aligned}
\\
\begin{aligned}\label{eq:Weak:Stokes:Eps:3}
\we(0) &= \wein && \textrm{ in }\Oe(0) 
\end{aligned}
	\end{align}
\end{subequations}
holds for a.e.~$(\varphi, \phi) \in L^2(0,T;H^1_{\Ge(t)} (\Oe(t); \R^d)) \times L^2(\Oe^T)$.

The space $L^2(0,T;H^1_{\Ge(t)}(\Oe(t); \R^d))$ has to be understood as the subset of \linebreak $L^2(0,T;H^1(\Omega; \R^d))$
of functions which are zero in $((0,T) \times \Omega) \setminus \Oe^T$. The time-derivative in $L^2(\Oe^T)$ has to be understood in the sense that the extension of $\ve$ by $0$ to $\Omega$ is in $H^1(0,T;L^2(\Omega; \R^d))$ and $\partial_t \ve$ is zero in $((0,T) \times \Omega) \setminus \Oe^T$.

\subsection*{Assumptions on the data}
Let $\fe \in L^2(\Oe^T;\R^d)$, $\vGe \in H^1(0,T;H^1(\Omega);\R^d)$, 
$\pbe \in L^2(\Oe^T)$ with $\nabla_x \pbe \in L^2(\Oe^T;\R^d)$ and $\vein \in H^1(\Oe(0))$. We assume that the initial values $\vein$ and boundary values $\vGe$ are compatible, i.e.~${\div(\wein)) = -\div(\vGe(0))}$ and $\wein|_{\Ge(0)} =0$ for $\wein = \vein - \vGe(0)$.

We assume that the data satisfy the following uniform bounds:
\begin{subequations}\label{eq:BoundsData}
\begin{align}
	\|\fe\|_{L^2(\Oe^T)} + \|\vein\|_{L^2(\Oe(0))} + \e\|\nabla \vein\|_{L^2(\Oe(0))} &\leq C \, ,
	\\
	\e^{-1}\|\vGe\|_{L^2((0,T) \times \Omega)} + \|\partial_x \vGe\|_{L^2((0,T) \times \Omega)} + \e \|\partial_x \partial_x \vGe\|_{L^2((0,T) \times \Omega)} &\leq C \, ,
	\\
	\e^{-1}\|\partial_t \vGe\|_{L^2((0,T) \times \Omega)} + \| \partial_x \partial_t\vGe\|_{L^2((0,T) \times \Omega)} + \e \|\partial_x \partial_x \partial_t \vGe\|_{L^2((0,T) \times \Omega)} &\leq C \, .	
\end{align}
\end{subequations}
The uniform estimates for $\vGe$ and its derivatives give the uniform estimate for the trace at $t= 0$ 
\begin{align*}
\e^{-1}	\|\hvGe(0)\|_{L^2(\Omega)} + \|\nabla \hvGe(0)\|_{L^2(\Omega)} \leq C 
\end{align*}
and, thus,
\begin{align*}
\|\wein\|_{L^2(\Oe(0))} + \e\|\nabla \wein\|_{L^2(\Oe(0))} &\leq C \, .
\end{align*}

In order to state the assumptions on their asymptotic behaviour, we extend $\fe$, $\pbe$ and $\vein$ as well as their derivatives by zero to $(0,T) \times \Omega$ and $\Omega$, respectively, which we denote by $\widetilde{\cdot}$.
We assume that there exists
$f\in L^2((0,T) \times \Omega; \R^d) $, ${\vGn \in L^2(\Omega; H^1(0,T;H^2_\#(Y;\R^d)))}$, $\pb \in L^2(0,T;H^1(\Omega))$, $\pbO \in L^2((0,T) \times \Omega; H^1_\#(Y))$, $\vnin \in L^2(\Omega;H^1_{\#}(Y;\R^d)))$ with $\vnin\xy = 0$ for $y \in \Ys(0,x)$ such that
\begin{subequations}\label{eq:Two-scale-conv:data}
\begin{align}
\begin{aligned}
&\widetilde{\fe} \tsw{2,2} \one{\Omega_0^T}f \, ,
	\qquad
	\e^{-1} {\vGe} \tsw{2,2}\vGn \, ,
	\qquad
	 {\partial_x \vGe} \tsw{2,2} \partial_y \vGn \, ,
	\qquad 
	\e \partial_x \partial_x \vGe \tsw{2,2} \partial_y \partial_y \vGn \, ,
	\qquad
	\\
 &\e^{-1} \partial_t \vGe \tsw{2,2}\partial_t \vGn \, ,
	\qquad
	\partial_x \partial_t\vGe \tsw{2,2} \partial_y \partial_t \vGn \, ,
	\qquad
	 \e \partial_x \partial_x \partial_t \vGe \tsw{2,2} \partial_y \partial_y \partial_t \vGn \, ,
	\\
	&\widetilde{\pbe} \tsw{2,2} \one{\Omega_0^T} \pb\, , 
	\qquad
	\widetilde{\nabla \pbe} \tsw{2,2} \one{\Omega_0^T} (\nabla_x \pb + \nabla_y \pbO)\, , 
	\\
	&\widetilde{\vein}(x) \tsw{2,2} \one{\Yp(0,x)}(y) \vnin \xy \, ,
	\qquad
	\e \widetilde{\partial_x \vein}(x) \tsw{2,2} \one{\Yp(0,x)}(y) \partial_y \vnin \xy \, .
\end{aligned}
\end{align}

We note that $\vGe(0)$ is of order $\e$ and $\nabla \vGe(0)$ is of order $1$. Thus, their contribution in the limit of $\wein$ vanishes and we get
\begin{align*}
&\widetilde{\wein}(x) \tsw{2,2} \one{\Yp(0,x)}(y) \hat{w}^\init_0 \xy \, ,
&
&\e \widetilde{\partial_x \wein}(x) \tsw{2,2} \one{\Yp(0,x)}(y) \partial_y \hat{w}^\init_0 \xy \, 
\end{align*}
for $\hat{w}_0^\init = \hat{v}_0^\init$.
\end{subequations}

\section{Transformation of the micromodel}\label{sec:Trafo-eps-scaled}
In this section, we transform the Stokes equations to the reference domain $(0,T) \times \Oe$ by means of the Piola transformation and $\psie$. Moreover, we transform the data and the assumptions on their uniform bounds and convergence to the reference domain.
We denote the transformed quantities by $\hat{\cdot}$, i.e.~we have the transformed unknowns
\begin{align*}
&\hve\tx \coloneqq \Ae\tx\ve(t,\psie\tx) \, , 
&&\hwe\tx \coloneqq \Ae\tx \we(t,\psie\tx) \, ,
\\
&\hpe\tx \coloneqq \pe(t,\psie\tx) \, , &&\hqe\tx \coloneqq \qe(t,\psie\tx)
\end{align*}
and the transformed data
\begin{align*}
&\hfe\tx \coloneqq \fe(t,\psie\tx) \, , 
\\
&\hvGe\tx \coloneqq \Ae\tx\vGe(t,\psie\tx) \, , 
&&\hpbe\tx \coloneqq \pbe(t,\psie\tx) \, ,
\\
&\hvein (x) \coloneqq \vein(\psie(0,x)) \, ,
&&\hwein (x) \coloneqq \wein(\psie(0,x)) \, . 
\end{align*}
The multiplication by $\Aem$ becomes useful for the derivation of the existence results of the microscopic transformed problem since it avoids time-dependent coefficients in the divergence condition. Moreover, for the limit process it becomes useful since it avoids microscopically oscillating coefficients in the divergence condition.

For the transformation of the normal stress boundary condition at $\Lambda(t)$, we note the following relation between the outer unit normal vector $\nu(\psie\tx)$ of $\Oe(t)$ and the outer unit normal vector $\hat{\nu}(x)$ of the reference coordinates $\Oe$.
\begin{align*}
\| \PemT\tx \hat{\nu}(x)\|^{-1} \PemT\tx \hat\nu(x) &= \nu(\psie\tx) && \textrm{ for every } \tx \in [0,T] \times \partial \Oe \,.
\end{align*}
Transforming \eqref{eq:Strong:Stokes:Eps:1} to the reference coordinates gives \eqref{eq:Strong:Trafo:Stokes:Eps:1}--\eqref{eq:Strong:Trafo:Stokes:Eps:Def:SymGradient}.
\begin{subequations}\label{eq:Strong:Trafo:Stokes:Eps}
	\begin{align}\label{eq:Strong:Trafo:Stokes:Eps:1}
				\partial_t (\Aem\hve) - \nabla^\top(\Aem\hve) \Pem\partial_t \psie - \Jem \mu \e^2 \div(\Ae 2\hat{\nabla}^\textrm{s}_\e \hve) + \PemT\nabla \hpe &= \hfe && \textrm{in } (0,T) \times \Oe \, ,
		\\\label{eq:Strong:Trafo:Stokes:Eps:2}
		\Jem\div(\hve) &= 0 && \textrm{in } (0,T) \times \Oe \, ,
		\\\label{eq:Strong:Trafo:Stokes:Eps:3}
		\hve &= \hvGe && \textrm{on } (0,T) \times\Ge \, ,
		\\\label{eq:Strong:Trafo:Stokes:Eps:4}
		\left( -\mu \e^2 2 \hat{\nabla}^\textrm{s}_\e \hve + \hpe \1 \right)\| \PemT\tx \hat{\nu}(x)\|^{-1} \PemT \hat{\nu} = \pbe \| \PemT\tx \hat{\nu}(x)\|^{-1} &\PemT \hat{\nu} && \textrm{on } (0,T) \times \Le \, ,
		\\\label{eq:Strong:Trafo:Stokes:Eps:5}
		\hve(0) &= \hvein && \textrm{in } \Oe \, ,
		\\\label{eq:Strong:Trafo:Stokes:Eps:Def:SymGradient}
		\hat{\nabla}^\textrm{s}_\e \hve \coloneqq (\PemT \nabla (\Aem\hve) + (\PemT \nabla(\Aem\hve) )^\top)/2& && \textrm{in } (0,T) \times \overline{\Oe} \, .
	\end{align}
\end{subequations}

In order to derive the weak form, we multiply \eqref{eq:Strong:Trafo:Stokes:Eps:1} by $\PeT$, \eqref{eq:Strong:Trafo:Stokes:Eps:2} by $\Je$ and \eqref{eq:Strong:Trafo:Stokes:Eps:3} by $\| \PemT\tx \hat{\nu}(x)\|$. 
We rewrite the resulting first term of $(\PeT\eqref{eq:Strong:Trafo:Stokes:Eps:1})$ by
\begin{align*}
	\PeT(\partial_t (\Aem\hve) = \partial_t (\PeT \Aem\hve) - \partial_t \PeT (\Aem\hve)
\end{align*}
and we get
\begin{align*}
\partial_t (\PeT \Aem\hve) - \partial_t \PeT (\Aem\hve) - \PeT\nabla^\top(\Aem\hve) \Pem\partial_t \psie&
\\
 - \AeT \mu \e^2 \div(\Ae 2\hat{\nabla}^\textrm{s}_\e \hve) + \nabla \hpe &= \PeT \hfe && \textrm{in } (0,T) \times \Oe \, ,
\end{align*}
Proceeding as in the derivation of the weak form for the untransformed equation, we obtain the following weak form \eqref{eq:WeakFormEpsOperator}:

Find $(\hwe, \hqe) \in L^2(0,T;H^1_{\Ge}(\Oe; \R^d)) \times L^2((0,T) \times \Oe)$ with $\partial_t (\PemT \Aem \hwe), \partial_t \hwe \in L^2((0,T) \times \Oe; \R^d)$ such that
\begin{subequations}\label{eq:WeakForm:Eps:Trafo}
\begin{align}
\begin{aligned}\label{eq:WeakForm:Eps:Trafo:1}
\intTOe &\partial_t (\PeT \Aem \hwe) \cdot \varphi - \left( \partial_t \PeT \Aem \hwe -\PeT \nabla^\top (\Aem\hwe)\Pem \partial_t \psie \right) \cdot \varphi 
\\
&
+ \mu \e^2 \Ae 2\hat{\nabla}^\textrm{s}_\e \hwe : \nabla \varphi 
- \hqe \div(\varphi) \dxt \, 
\\
=
\intTOe& \PeT \hfe \cdot \varphi - \partial_t (\PeT \Aem \hvGe) \cdot \varphi -
 \nabla \hpbe \cdot \varphi 
\\
&
+ \left( \partial_t \PeT \Aem \hvGe -\PeT \nabla^\top (\Aem\hvGe)\Pem \partial_t \psie \right) \cdot \varphi -
 \mu \e^2 \Ae 2\hat{\nabla}^\textrm{s}_\e :\hvGe\nabla \varphi \dxt \,  ,
\end{aligned}
\\
\begin{aligned}\label{eq:WeakForm:Eps:Trafo:2}
\intTOe \phi \div(\hwe) \dxt &= - \intTOe \phi \div(\hvGe) \dxt \, ,
\end{aligned}
\\
\begin{aligned}\label{eq:WeakForm:Eps:Trafo:3}
\hwe(0) &= \hwein
\end{aligned}
\end{align}
\end{subequations}
for all $(\varphi, \phi) \in L^2(0,T;H^1_{\Ge}(\Oe;\R^d)) \times L^2(0,T;L^2(\Oe))$.

\paragraph{Transformation of the data:}
For the transformed data $\hfe$, $\hvGe$, $\hpbe$ and $\hvein$, we can transfer the uniform bounds with Lemma \ref{lem:BoundsUnderTransformation} and obtain:
\begin{lemma}[Uniform bounds of the transformed data]\label{lem:Est:Data:Trafo}
Let $\hfe$, $\hvGe$, $\hpb$ and $\wein$ by given as above. Then there exists a constant $C>0$ such that
\begin{subequations}\label{eq:BoundsData:Trafo}
	\begin{align}\label{eq:BoundsData:f+win:Trafo}
		\|\hfe\|_{L^2((0,T) \times \Oe)} + \|\hwein\|_{L^2(\Oe)} + \e\|\nabla \hwein\|_{L^2(\Oe)} &\leq C \, ,
		\\\label{eq:BoundsData:vGe:Trafo}
		\e^{-1}\|\hvGe\|_{L^2((0,T) \times \Oe)} + \|\partial_x \hvGe\|_{L^2((0,T) \times \Oe)} + \e \|\partial_x \partial_x \hvGe\|_{L^2((0,T) \times \Oe)} &\leq C \, ,
		\\\label{eq:BoundsData:dt-vGe:Trafo}
		\e^{-1}\|\partial_t \hvGe\|_{L^2((0,T) \times \Oe)} + \|\partial_x \partial_t\hvGe\|_{L^2((0,T) \times \Oe)} &\leq C \, .
	\end{align}
\end{subequations}
Moreover, $\hwein$ is compatible i.e.~$\hwein \in H^1_{\Ge}(\Oe;\R^d)$ and $\div(\hwein) = -\div(\hvGe)$. 
\end{lemma}
\begin{proof}
The uniform estimates for $\hfe$, $\hwein$ and $\nabla \hwein$ can be deduced with Lemma~\ref{lem:BoundsUnderTransformation} and Remark \ref{rem:two-scale-trafo:FixedPointTime}.

To derive the uniform estimates on $\hvGe$, we note that the uniform estimates on $\vGe$ and Lemma~\ref{lem:BoundsUnderTransformation} provide a uniform estimate for $	\e^{-1}\|\vGe(t,\psie\tx)\|_{L^2((0,T) \times \Omega)}$, $\|\partial_x \vGe(t,\psie\tx)\|_{L^2((0,T) \times \Omega)}$ and $\e \|\partial_x \partial_x \vGe(t,\psie\tx)\|_{L^2((0,T) \times \Omega)}$\,.
We apply the product rule on $\hvGe$ and together with the estimates on $\Ae$ and its spatial derivatives given in Lemma \ref{lem:Est:Psie}, we get \eqref{eq:BoundsData:vGe:Trafo}.

In order to derive \eqref{eq:BoundsData:dt-vGe:Trafo}, we note that 
\begin{align*}
	\partial_t \hvGe\tx =& \partial_t \Ae\tx \vGe(t,\psie\tx) + \Ae\tx \partial_t \vGe(t,\psie\tx) \\
 &+ \Ae\tx \partial_x \vGe(t,\psie\tx) \partial_t \psie\tx \,
\end{align*}
and, hence, the estimates on $\Ae$ given in Lemma \ref{lem:Est:Psie} together with the estimates on $\hvGe$ provide the uniform bound on $\e^{-1}\partial_t \hvGe$.
Taking the derivative with respect to $x$ in the previous equation and employing the product rule, one can similarly deduce the uniform estimate on $\partial_x \partial_t \hvGe$.

The compatibility of the initial values, i.e.~~${\div(\hwein)) = -\div(\hvGe(0))}$ and $\hwein|_{\Ge(0)} =0$ for $\hwein = \hvein - \hvGe(0)$ is preserved under the transformation.
\end{proof}

With Lemma \ref{lem:TwoScaleEquiv}, we can also transform the two-scale convergences of the data arguing similarly as in the proof of Lemma \ref{lem:Est:Data:Trafo}. We get for $\hat{f} = f \in L^2((0,T) \times \Omega; \R^d)$, ${\hvGn \in H^1(0,T;L^2(\Omega; H^1_\#(Y;\R^d)}$ with $\hvGn\txy = \vGn(t,x,\psin\txy)$, $\hpb = \pb \in L^2(0,T;H^1(\Omega))$, $\hpbO \in L^2((0,T) \times \Omega; H^1_\#(Y))$ with $\hpbO= \pbO(t,x,\psin\txy) + (\psin\txy-y) \cdot \nabla \pb$, $\hvn^\init \in L^2(\Omega;H^1(Y)$ with $\hvn^\init\xy= \vnin(x,\psin(0,x,y))$ that
\begin{subequations}\label{eq:Two-scale-conv:data:Trafo}
	\begin{align}
 \begin{aligned}
     &\widetilde{\hfe} \tsw{2,2} \one{\Yp} f \, , \quad
		\e^{-1} \widetilde{\hvGe}\tsw{2,2} \one{\Yp}\hvGn \, ,
		\quad
	   \partial_x \hvGe \tsw{2,2} \one{\Yp}\partial_y \hvGn \, ,
		\quad
		\e \partial_x \partial_x \vGe \tsw{2,2} \partial_y \partial_y \vGn \, ,
		\\
		&\e^{-1} \partial_t \hvGe \tsw{2,2} \hvGe\, ,
		\quad
	    \partial_x \partial_t\vGe \tsw{2,2} \partial_y \partial_t \vGn \, ,
		\quad
		\e \hvGe \tsw{2,2} \hvGn \, ,
		\\
		&\hpbe \tsw{2,2} \one{\Yp} \hpb, ,
		\quad
        \widetilde{\nabla \hpbe} \tsw{2,2} \one{\Yp} (\nabla_x \hpb+ \nabla_y \hpbO)\, , 
		\\
		&\widetilde{\hwein}(x) \tsw{2,2} \one{\Yp}(y) \hat{w}_0^\init\xy \, ,
		\qquad
		\widetilde{\e \partial_x \hwein}(x) \tsw{2,2} \one{\Yp}(y) \partial_y \hat{w}_0^\init\xy \, . 
 \end{aligned}
	\end{align}
\end{subequations}
As in the untransformed case, we get $\hat{w}_0^\init = \hat{v}_0^\init$.

\section{Existence and a-priori estimates for the microscopic problem}\label{sec:Existence:eps-scaled}
In this section, we show the existence and uniqueness of a solution of \eqref{eq:WeakForm:Eps:Trafo}. Moreover, we derive the uniform a-priori estimates \eqref{eq:AprioirBounds:Stokes} for the solution.

\begin{thm}[Existence, uniqueness and a-priori estimates of the solution of the Stokes equations]\label{thm:Existence:Eps}
	For every $\e >0$, there exists a unique solution $(\hwe, \hqe) \in L^2(0,T;H^1_{\Ge}(\Oe; \R^d)) \times L^2((0,T) \times \Oe)$ with $\partial_t \hwe \in L^2((0,T) \times \Oe; \R^d)$ of \eqref{eq:WeakForm:Eps:Trafo}.
	Moreover, there exists a constant $C >0$ such that for every $\e >0$
	\begin{align}\label{eq:AprioirBounds:Stokes}
		\| \hwe \|_{L^2((0,T) \times \Oe)} + \e \|\nabla \hwe \|_{L^2((0,T) \times \Oe)} + \|\hqe \|_{L^2((0,T) \times \Oe)} \leq C \, .
	\end{align}
\end{thm}

In order to derive the existence and uniqueness of a solution of \eqref{eq:WeakForm:Eps:Trafo}, we use a generic existence result for differential--algebraic operator equations from \cite{Zim21}, which is given in Proposition \ref{prop:Generic:Existence}. For this, we use the following notation.
Following \cite{HP57}, we say for Banach spaces $X,Y$ an operator-valued function $\mathrm{A} : [0,T] \to \Lin(X,Y)$ is strongly measurable if $t \mapsto \mathcal{A}(t) x$ is Bochner-measurable for every $x \in X$\,.
Let $\mathrm{A} : [0,T] \to \Lin(X,Y)$ be strongly measurable, assume that $t \mapsto \mathrm{A}(t)x$ has a $k$-th generalised derivative $\tfrac{\dd}{\dt}(\mathcal{A}(\cdot) x) $ for some $k\in \N$ and every $x \in X$. Then, the $k$-th derivative $\mathcal{A}^{(k)} \colon [0,T] \to \Lin(X,Y)$ of $\mathcal{A}$ is $\mathcal{A}^{(k)}(t)x \coloneqq \tfrac{\dd}{\dt}(\mathcal{A}(\cdot) x)$\,.
Following the notation of \cite{BvN10}, for Banach spaces $X$, $Y$ with $X$ separable, we say $\mathcal{A} [0,T] \mapsto \Lin(X,Y)$ belongs to $L^p[0,T;\Lin(X,Y)]$ for $p \in[1,\infty]$ if $\mathcal{A}$ is strongly measurable, $t \mapsto \|\mathcal{A}(t)\|_{\Lin(X,Y)}$ is Lebesgue measurable and
${\|\mathcal{A}\|_{L^p[0,T;\Lin(X,Y)]} \coloneqq \| \|\mathcal{A}(\cdot)\|_{\Lin(X,Y)} \|_{L^p(0,T)} < \infty}$\,.
Furthermore, $\mathcal{A} [0,T] \mapsto \Lin(X,Y)$ belongs to $W^{k,p}[0,T;\Lin(X,Y)]$ for $p \in [1,\infty]$, if $\mathcal{A}^{(k)} \in L^p[0,T;\Lin(X,Y)]$ for every $i \in \{0, \dots, k\}$\,. We write $H^k[0,T;\Lin(X,Y)] \coloneqq W^{k,2}[0,T;\Lin(X,Y)]$\,.

\begin{prop}[Existence result for operator differential--algebraic equations]\label{prop:Generic:Existence}
Let $V$, $H$, $Q$ be separable Hilbert spaces. Assume $V$, $H$, $V^*$ form a Gelfand triple with embedding constant $\CVH$ of $V$ in $H$. Let $T >0$, $\mathcal{M} \in H^1[0,T;\Lin(H,H^*)]$, $\mathcal{A} \in L^\infty[0,T;\Lin(V,V^*)]$, $B\in \Lin(V,Q^*)$ and $F = F^{(1)} + F^{(2)}$ for $F^{(1)}\in L^2(0,T;H^*)$, $F^{(2)} \in W^{1,1}(0,T;V^*)$, $G \in H^1(0,T;Q^*)$ and $v^\init \in V$ with $\mathcal{B}v = G(0)$. Assume that $\mathcal{M}$ is self-adjoint and uniformly elliptic with constant $\mu_{\mathcal{M}}>0$, i.e.~for every $t \in [0,T]$ and every $v \in H$ 
\begin{align*}
\mathcal{M}(t) (v,v) \geq \mu_{\mathcal{M}} \|v\|_H \, ,
\end{align*}
assume $\mathcal{A}$ can be decomposed in $\mathcal{A} = \mathcal{A}^{(1)} + \mathcal{A}^{(2)}$ with $\mathcal{A}^{(1)} \in L^\infty[0,T;\Lin(V;V^*)]$ and $\mathcal{A}^{(2)} \in L^\infty[0,T;\Lin(V;H^*)]$ such that $\mathcal{A}^{(1)}$ is self-adjoint and there exists constants $\mu_{\mathcal{A}^{(1)}}, \kappa_{\mathcal{A}^{(1)}}$ such that for a.e.~$t\in(0,T)$ and every $v \in \operatorname{ker}(\mathcal{B})$
\begin{align}\label{eq:Garding:A}
	\mathcal{A}^{(1)}(t)(v,v) \geq \mu_{\mathcal{A}^{(1)}}\|v\|_V^2 - \kappa_{\mathcal{A}^{(1)}} \|v\|_H^2\, .
\end{align}
Moreover, we assume that $\mathcal{B}$ is inf--sup stable with constant $\mu_\mathcal{B}$, i.e.~
\begin{align*}
\adjustlimits\inf_{q \in M\setminus\{0\}} \sup_{v \in V \setminus \{0\}} \frac{\mathcal{B}(v,q)}{\|q\|_Q \|v\|_V} \geq \mu_{\mathcal{B}} \,.
\end{align*}
Then, there exists a unique $(v,p) \in C([0,T];V) \cap H^1(0,T;H) \times L^2(0,T;Q)$ such that for a.e.~$t\in (0,T)$
\begin{subequations}
	\begin{align}\label{eq:Generic:Operator-Problem}
		\frac{\dd}{\dt} (\mathcal{M}(t) v(t)) ) + (\mathcal{A}(t) - \frac{1}{2} \overset{.}{\mathcal{M}}(t)) v(t)) - \mathcal{B}^*p(t) &= F(t) &&\textrm{ in } V^* \, ,
		\\
		\mathcal{B} v(t)&= G(t) && \textrm{ in } Q^* \, ,
		\\
		v(0) &= v^\init && \textrm{ in } V \, .
	\end{align}
\end{subequations}
Moreover, there exists a constant $C$, which depends only on $T$, $\|\mathrm{M}\|_{H^1[0,T;\Lin(H,H^*)]}$, $\|\mathrm{A}^{(1)}\|_{L^\infty[0,T;\Lin(V,V^*)]}$, $\|\mathrm{A}^{(2)}\|_{L^\infty[0,T;\Lin(V,H^*)]}$, $\mu_{\mathcal{M}}$, $\mu_{\mathcal{A}^{(1)}}$, $\kappa_{\mathcal{A}^{(1)}}$, $\mu_{\mathcal{B}}$, $\CVH$ such that
\begin{align*}
\|v\|_{C([0,T];V)} + &\|v\|_{H^1((0,T);H)}+ \|q\|_{L^2(0,T;Q)} 
\\
&\leq C (\|F^{(1)}\|_{L^2(0,T;H^*)} + \|F^{(2)}\|_{W^{1,1}(0,T;V^*)} + \|G\|_{H^1(0,T;Q^*)} + \|u_0\|_V ) \, .
\end{align*}
\end{prop}
\begin{proof}
For the case that $\mathcal{A}^{(1)}$ is uniformly elliptic and $F_2 =0$ the result is shown in \cite[Theorem 7.24]{Zim21}. Remark \cite[Theorem 7.25]{Zim21} extends it to the case $F_2 \neq 0$. The case that $\mathcal{A}^{(1)}$ satisfies only the weaker Garding inequality \eqref{eq:Garding:A} can be reduced to this case by reformulating and rescaling \eqref{eq:Generic:Operator-Problem} (see \cite[Remark~7.1]{Zim21})
\end{proof}

We reformulate the weak form of the Stokes equations \eqref{eq:WeakForm:Eps:Trafo} in the generic setting of Proposition \ref{prop:Generic:Existence}. We account for the $\e$-parameter by means of the subscript $\e$.

Let $V_\e \coloneqq H^1_\Ge(\Oe; \R^d)$, $H_\e\coloneqq L^2(\Oe; \R^d)$, $Q_\e \coloneqq L^2(\Oe)$ with the norms
\begin{align*}
\|v \|_{V_\e} &\coloneqq \e \|\nabla v\|_{L^2(\Oe)} && \textrm{for } v\in V_\e \, ,
\\
\|v \|_{H_\e} &\coloneqq \|v\|_{L^2(\Oe)} && \textrm{for } v\in H_\e \, ,
\\
\|q \|_{Q_\e} &\coloneqq \|q\|_{L^2(\Oe)} && \textrm{for } q\in Q_\e \, .
\end{align*}
We define the operators $\mathcal{M}_\e \in H^1[0,T; \Lin(H_\e, H_\e^*)]$, $\mathcal{A}_\e \in L^\infty[0,T; \Lin(V_\e, V_\e^*)]$ with 
$\mathcal{A}_\e = \mathcal{A}_\e^{(1)} + \mathcal{A}_\e^{(2)}$ for
$\mathcal{A}_\e^{(1)} \in H^1[0,T; \Lin({V}_\e, {V}_\e^*)]$ and $\mathcal{A}_\e^{(2)} \in H^1[0,T; \Lin({V}_\e, {H}_\e^*)]$, $\mathcal{B}_\e \in \Lin(V_\e, Q_\e^*)$ as well as the right-hand sides $F_\e^{(1)} \in L^2(0,T;H^*_\e)$, $F_\e^{(2)} \in W^{1,1}(0,T;V^*_\e)$ and $G_\e \in H^1(0,T; Q_\e^*)$ by
\begin{align*}
	\mathcal{M}_\e(t) (u,v) &\coloneqq \int\limits_{\Oe} \PemT(t) \Aem(t) u \cdot v \dx && \textrm{for } u,v \in H_\e \, ,
	\\
	\mathcal{A}_\e(t) (u,v) &\coloneqq \mathcal{A}_\e^{(1)}(t) (u,v) + \mathcal{A}_\e^{(2)}(t) (u,v) && \textrm{for } u,v \in V_\e \, ,
	\\
	\mathcal{A}_\e^{(1)}(t) (u,v) &\coloneqq \int\limits_{\Oe} \mu \e^2 \Ae(t) 2\hat{\nabla}^\textrm{s}_\e(t) v: \nabla (\Aem(t) v) \dx && \textrm{for } u,v \in V_\e \, ,
	\\
	\hat{\nabla}^\textrm{s}_\e(t) v &\coloneqq (\PemT(t) \nabla (\Aem(t) u) + (\PemT(t) \nabla(\Aem(t)u) )^\top) /2  &&\textrm{for } u,v \in V_\e \, ,
	\\
	\mathcal{A}_\e^{(2)}(t) (u,v) &\coloneqq - \intOe \big( \partial_t \PeT(t) \Aem(t) u -\PeT(t) \nabla^\top (\Aem(t) u)\Pem(t) \partial_t \psie(t) \big) \cdot v \dxt \hspace{-3.5cm} &&
	\\
	&\qquad	+ \tfrac{1}{2} \overset{.}{\mathcal{M}}_\e(t)(u,v) && \textrm{for } u \in V_\e, \, v \in H_\e \, ,
	\\
	&
	\\
	\mathcal{B}_\e(v,q) &\coloneqq \intTOe q \div(v) \dxt && \textrm{for } v \in V_\e, \, q \in Q_\e \, ,
	\\
	F_\e^{(1)}(t)(\varphi) &\coloneqq \intTOe (\PeT(t) \hfe(t) - \nabla \hpbe(t)) \cdot v \dx
 \\
 &\qquad -	\mathcal{M}_\e(t) (\hvGe (t), v)  && \textrm{for } v \in H_\e \,,
	\\
	F_\e^{(2)}(t)(\varphi) &\coloneqq   - \mathcal{A}_\e(t) (\hvGe (t), v) && \textrm{for } v \in V_\e \,,
	\\
	G_\e(t)(q) &\coloneqq -	\mathcal{B}_\e(\hvGe (t), q ) && \textrm{for } q \in Q_\e \,.
\end{align*}
Thus, we have rewritten the weak form \eqref{eq:WeakForm:Eps:Trafo} in the generic setting of Proposition~\ref{prop:Generic:Existence}:
\begin{subequations}
\begin{align}\label{eq:WeakFormEpsOperator}
	\frac{\dd}{\dt} (\mathcal{M}_\e(t) \hve(t) ) + (\mathcal{A}_\e(t) - \frac{1}{2} \overset{.}{\mathcal{M}}_\e(t)) \hve(t)) - \mathcal{B}_\e^* \hqe(t) &= F_\e(t) &&\textrm{ in } V^*_\e \, 
	\\
	\mathcal{B}_\e \hqe(t) &= G_\e(t) && \textrm{ in } Q_\e^* \,,
	\\
	\ve(0) & = \vein && \textrm{ in } V_\e^* \, .
\end{align}    
\end{subequations}
In order to deduce the uniform bounds \eqref{eq:AprioirBounds:Stokes} from Proposition \ref{prop:Generic:Existence}, it is essential that we estimate the embedding constant $\CVH$, the Garding inequality constants $\mu_{\mathcal{A}_\e^{(1)}}$ and $\kappa_{\mathcal{A}_\e^{(1)}}$ as well as the inf--sup constant $\mu_{\mathcal{B}_\e}$ uniformly.

We obtain a uniform estimate for the embedding constant $\CVH$ from the following $\e$-scaled Poincar\'e inequality.
\begin{lemma}[$\e$-scaled Poincar\'e inequality]\label{lem:EpsPoincare}
There exists a constant $c_\mathrm{P}$ such that for every $v \in H^1_{\Ge}(\Oe;\R^d)$
\begin{align*}
\|v\|_{L^2(\Oe)} \leq \e c_\mathrm{P} \|\nabla v\|_{L^2(\Oe)} \, .
\end{align*}
\begin{proof}
Lemma \ref{lem:EpsPoincare} is a standard result and can be shown by covering $\Oe$ with
$\e$-scaled copies of $\Yp$. Scaling them on $\Yp$ and applying the Poincar\'e
inequality for piecewise zero boundary values there and scaling back yields the estimate.
\end{proof}
\end{lemma}
The uniform inf--sup constant can be deduced from the following $\e$-scaled right-inverse of the divergence operator.
\begin{lemma}[$\e$-scaled right-inverse of the divergence operator]\label{lem:div:eps:inverse}
There exists a family of linear continuous operators $\div_\e^{-1} : L^2(\Oe) \to H^1_\Ge(\Oe;\R^d)$, which are right inverse to the divergence, i.e.~$\div \circ \div_\e^{-1} = \operatorname{id}_{L^2(\Oe)}$,
and there is a constant $C>0$ such that for all $f\in L^2(\Oe)$
\begin{align*}
\|\div^{-1}_\e (f)\|_{L^2(\Oe)} + \e \|\nabla (\div^{-1}_\e (f))\|_{L^2(\Oe)} \leq \|f\|_{L^2(\Oe)} \, .
\end{align*}
\end{lemma}
\begin{lemma}
	Lemma \ref{lem:div:eps:inverse} was shown in \cite[Lemma 3.12]{JDE24} employing the extension operators of \cite{Tar80, All89}.
\end{lemma}
For the Garding inequality \eqref{eq:Garding:A}, it becomes crucial to estimate the symmetrised gradient $\hat{\nabla}_\e^\mathrm{s}$. Here we use the following Korn inequality
\begin{lemma}[Korn-type inequality for two-scale transformations]\label{lem:Korn}
There exists a constant $\alpha$ such that for every $\e>0$, a.e.~$t\in (0,T)$ and every $v \in H^1{\Ge}(\Oe;\R^d)$
\begin{align*}
\alpha \|v\|_{L^2(\Oe)}^2 \leq \| \PemT(t) \nabla v + \PemT(t) \nabla v \|^2 \, .
\end{align*}
\end{lemma}
\begin{proof}
A proof is given in \cite[Lemma 3.6]{JDE24}.
\end{proof}
\begin{proof}[Proof of Theorem \ref{thm:Existence:Eps}]
We show Theorem \ref{thm:Existence:Eps}, by means of Theorem~\ref{thm:Existence:Eps}. In order to derive the estimate \eqref{eq:AprioirBounds:Stokes} we show a uniform estimate for the continuity constant $\CVH$ of the embedding $V_\e \to H_\e$, uniform bounds for the operators $\mathcal{A}_\e$, $\mathcal{A}_\e^{(1)}$, $\mathcal{A}_\e^{(1)}$, $\mathcal{A}_\e^{(2)}$, the right-hand sides $F_\e^{(1)}$, $F_\e^{(2)}$ and $G_\e$, the initial value $\ve$ as well as the ellipticity constant $\mu_{\mathcal{M}_\e}$ of $\mathcal{M}_\e$, the Garding inequality constants $\mu_{\mathcal{A}_\e^{(1)}}$ and $\kappa_{\mathcal{A}_\e^{(1)}}$ of $\mathcal{A}_\e^{(1)}$ and the inf--sup constant of $\mu_{\mathcal{B}_\e}$ of $\mathcal{B}_\e$.

For the following estimates on the operators, the uniform estimate for the transformation coefficients becomes crucial, which is given in Lemma \ref{lem:Est:Psie}.
\begin{itemize}
\item Embedding constant $\CVH$: 
The Poincar\'e estimate from Lemma \ref{lem:EpsPoincare} provides a uniform embedding constant $\CVH$, i.e.~for every $v \in V_\e$, it holds
\begin{align}\label{eq:PoincareHeVe}
	\| v\|_{H_\e} = \|v\|_{L^2(\Oe)} \leq \CVH \e\|\nabla v\|_{L^2(\Oe)} = \CVH \| v\|_{V_\e} \, .
\end{align}

\item Operator $\mathcal{M}_\e$:
Noting that $\Aem = \Jem \Pe$, we obtain for every $t\in [0,T]$ that $\mathcal{M}_\e(t)$ is self-adjoint from
\begin{align*}
\mathcal{M}_\e(t)(v,w) &= \intOe \PeT(t) \Aem(t) v \cdot w \dx = \intOe v \cdot ( \PeT(t)\Jem \Pe(t))^\top w \dt = 
\\
& =
\intOe v \cdot \PeT(t) \Aem(t) w \dt = \mathcal{M}_\e(t)(w,v) \, .
\end{align*}
We note that for a.e.~$t\in [0,T]$ and $u,v \in H_\e$
\begin{align*}
\overset{.}{\mathcal{M}}_\e(t)(u,v) = \intOe \partial_t (\PemT(t) \AemT(t)) u \cdot v \dx \, .
\end{align*}
Since $\PemT$, $\AemT$ and $\partial_t \PemT$, $\partial_t \AemT$ are bounded in $L^\infty((0,T) \times \Oe)^{d\times d}$, we can estimate with the H\"older inequality
\begin{align*}
\|\mathcal{M}_\e \|_{L^2[0,T;\Lin(H_\e, H_\e^*)]} 
&\leq C \|\PemT\|_{L^\infty((0,T) \times \Oe)} \|\AemT\|_{L^\infty((0,T) \times \Oe)} \leq C \, ,
\\
\|\overset{.}{\mathcal{M}_\e} \|_{L^2[0,T;\Lin(H_\e, H_\e^*)]} 
&\leq C \big( \|\PeT\|_{L^\infty((0,T) \times \Oe)} \|\partial_t \AemT\|_{L^\infty((0,T) \times \Oe)} 
\\
&\quad + \|\partial_t \PeT\|_{L^\infty((0,T) \times \Oe)} \|\AemT\|_{L^\infty((0,T) \times \Oe)} \big) \leq C \, .
\end{align*}
In order to show the ellipticity of $\mathcal{M}_\e(t)$, we rewrite 
\begin{align*}
\mathcal{M}_\e(t)(v, v) &= (\PeT(t) \Jem(t) \Pe(t) v \cdot v)_{H_\e} = (\Je^{-1/2}(t) \Pe(t) v \cdot \Je^{-1/2}(t) v)_{H_\e} 
\\
&= \| \Je^{-1/2}(t)\Pe(t) v \|_{H_\e}^2
\end{align*}
for $v \in H_\e$ and use the uniform essential bounds of $\Je$ and $\Pem$ to deduce with the H\"older inequality
\begin{align*}
	\|v \|^2_{H_\e}& \leq \|\Je^{1/2}(t) \Pem(t) \Je^{-1/2}(t)\Pe(t) v\|^2_{L^2(\Oe)}
	\\
	&\leq \|\Je^{1/2}(t) \Pem(t)\|^2_{L^\infty(\Oe)} \| \Je^{-1/2}(t)\Pe(t) v \|_{H_\e}^2 
	\leq C \mathcal{M}_\e(t)(v, v) \, .
\end{align*}
Choosing $\mu_{\mathcal{M}_\e} = C^{-1}$ gives a uniform estimate for the ellipticity constant. 

\item Operators $\mathcal{A}_\e^{(1)}$ and $\mathcal{A}_\e^{(2)}$:
We can estimate with the H\"older inequality, the product rule and the uniform bounds of the coefficients and the Poincar\'e estimate \eqref{eq:PoincareHeVe} 
\begin{align*}
\|&\mathcal{A}_\e^{(1)}(t)(u,v)\|_{L^\infty(0,T)} 
\\
&= \mu 2 \| \Ae \|_{L^\infty((0,T) \times \Oe)} \| \PemT \|_{L^\infty((0,T) \times \Oe)} \e \|\nabla (\Aem u)\|_{L^\infty(0,T;L^2(\Oe))} \| \nabla v\|_{L^2(\Oe)}
	\\
	&\leq
	 C\big( \e \|\nabla \Aem\|_{L^\infty((0,T) \times \Oe)} \|u\|_{L^2(\Oe)}
 +
 \e \|\Aem\|_{L^\infty((0,T) \times \Oe)} \|\nabla u\|_{L^2(\Oe)} \big)\e \| \nabla v\|_{L^2(\Oe)}
 \\
 &\leq
 C \|u\|_{V_\e} \|v\|_{V_\e} \, .
\end{align*}
We note that for a symmetric matrix $A$ and a non symmetric matrix $B$ it holds $A: B = A : (B+B^\top)/2$ and, thus, we can rewrite 
\begin{align*}
	\mathcal{A}_\e^{(1)}(t) (u,v) &= \intOe\mu \e^2 \Je 2\hat{\nabla}^\textrm{s}_\e(t) u: \PemT(t)\nabla(\Aem(t) v) \dx 
	\\
	&=
	\int\limits_{\Oe} \mu \e^2 \Je 2\hat{\nabla}^\textrm{s}_\e(t) u: \hat{\nabla}^\textrm{s}_\e(t) v \dx 
	= 
	\mathcal{A}_\e^{(1)}(t) (v,w) \, ,
\end{align*}
which shows that $\mathcal{A}_\e^{(1)}(t)$ is self-adjoint.
Using this rewriting and the H\"older inequality, we get
\begin{align*}
\| \hat{\nabla}^\textrm{s}_\e(t) v\|_{L^2(\Oe)}^2 &\leq \tfrac{1}{2}\|\Je^{-1/2}(t) \|_{L^\infty(\Oe)}^2 \| \Je^{1/2}(t) \sqrt{2}\hat{\nabla}^\textrm{s}_\e(t) v\|_{L^2(\Oe)}^2
\\
& = \tfrac{1}{2}\|\Je^{-1/2}(t) \|_{L^\infty(\Oe)}^2	\mathcal{A}_\e^{(1)}(t) (v,v) \, .
\end{align*}
We apply the Korn-type inequality of Lemma \ref{lem:Korn} on $(\Aem(t) v)$ in order to estimate the left-hand side from below and get
\begin{align*}
\alpha \e^2 \| \nabla( \Aem(t) v)\|_{L^2(\Oe)}^2 \leq \tfrac{1}{2}\|\Je^{-1/2}(t) \|_{L^\infty(\Oe)}^2	\mathcal{A}_\e^{(1)}(t) (v,v) \, .
\end{align*}
Moreover, with the H\"older inequality, the uniform essential boundedness of $\Ae(t)$ and $\e \partial_x \Ae$ and the Young inequality, we get a constant $\delta >0$ such that
\begin{align*}
\e^2\|\nabla v\|_{L^2(\Oe)}^2
&= \e^2 \|\nabla (\Ae(t) \Aem(t) v)\|_{L^2(\Oe)}^2 
\leq \e^2\|\nabla (\Ae(t) \Aem(t) v)\|_{L^2(\Oe)}^2 
\\
&\leq (\e\|(\Ae(t) \partial_x (\Aem(t) v))\|_{L^2(\Oe)} + \e\| (\Aem(t) v)^\top \partial_x \Ae(t))\|_{L^2(\Oe)} )^2
\\
&\leq (C \e \|\nabla (\Aem(t) v))\|_{L^2(\Oe)} + C \|(\Aem(t) v))\|_{L^2(\Oe)} )^2
\\
&\leq
\e^2\delta\|\nabla (\Aem(t) v))\|_{L^2(\Oe)}^2 + \tfrac{C}{\delta} \|v\|_{L^2(\Oe)}^2\,.
\end{align*}
Choosing $\delta = 2\|\Je^{-1/2}(t) \|_{L^\infty(\Oe)}^{-2}$ and combining the last two equations gives
\begin{align*}
	\alpha \|v\|_{V_\e}^2 - C \|v\|_{H_\e}^2 \leq 
\alpha \|\nabla v\|_{L^2(\Oe)}^2 - \tfrac{C}{\delta} \|v\|_{L^2(\Oe)}^2 \leq	\mathcal{A}_\e^{(1)}(t) (v,v) \, ,
\end{align*}
which shows that $\mathcal{A}_\e^{(1)}(t)$ satisfies a uniform Garding inequality with time- and $\e$-independent constants.

The uniform estimate on $\mathcal{A}_\e^{(2)}(t)$ can be shown by similar computations as for the estimate of  and $\mathcal{A}_\e^{(1)}(t)$ above. One only has to be aware of the fact that $\partial_t \psie \leq \e C$ and thus, one can compensate the factor $\e^{-1}$ which arises in the estimates of $\partial_x \Aem$ and $\nabla u$.

\item Operator $\mathcal{B}_\e$:
Lemma \ref{lem:div:eps:inverse} provides the uniform inf--sup constant $\mu_{\mathcal{B}}$.
Moreover, we get
\begin{align*}
|\mathcal{B}_\e(v,q)| = \intOe q \div(v) \dx \leq C \|q\|_{L^2(\Oe)} \|\nabla v\|_{L^2(\Oe)} = \e^{-1}\|q\|_{H_\e} \|v\|_{V_\e} \, .
\end{align*}
Indeed this does not provide a uniform on $\ |\mathcal{B}_\e(v,q)\|_{\Lin(V,Q^*)}$; however, the bounds on the solutions are independent on $\|\mathcal{B}_\e(v,q)\|_{\Lin(V,Q^*)}$.

\item Estimates for the right-hand sides $F_\e^{(1)}$, $F_\e^{(2)}$ and $G_\e$:
The estimates for the right-hand side $F_\e^{(1)}$ and  $F_\e^{(2)}$  can be deduced from the uniform bound of $\hfe$, the uniform bounds of $\hvGe$ and its derivatives, which are given in \eqref{eq:BoundsData:Trafo}, together with similar estimates as above for the operators.

The estimate for $G_\e$ can be deduced from the estimate on $|\mathcal{B}_\e(v,q)| $ from above and the uniform bound of $\e \nabla \hvGe$, which compensates the factor $\e^{-1}$, which arises in the estimate of $|\mathcal{B}_\e(v,q)|$.

\item initial values: The estimates and compatibility of the initial values is shown in Lemma \ref{lem:Est:Data:Trafo}.
\end{itemize}
\end{proof}

\section{Homogenisation of the substitute problem}\label{sec:Limit-process}
In this section, we pass to the homogenisation limit for the solution $(\hwe, \hqe)$ of \eqref{eq:WeakForm:Eps:Trafo}.
In order to state the convergence, we extend $\hwe$ and $\hqe$ to $\Omega$. We extend $\hwe$ by $0$, which we denote by $\widetilde{w_\e}$ and $\hqe$ by its cell-wise mean value, which we denote by $\hQe$, i.e.~
\begin{align}\label{eq:Extension:hQe}
\hQe \tx \coloneqq 
\begin{cases}
\hqe\tx &\textrm{for } x\in \Oe \, ,
\\
\frac{1}{|\e\Yp|}\int\limits_{\e(k + \Yp)}\hqe\tx &\textrm{for } x\in \Omega \cap \e(k + \Ys) \, \textrm{for } k \in \Ke \, .
\end{cases}
\end{align}

The physically more relevant quantity is $\hve$ and not $\hwe$. It is given by $\hve = \hwe + \hvGe$ in $\Oe$ and extended by $0$ in the solid domain $\Oes$, which we denote by $\widetilde{\hve}$. This extension of $\hve$ is not $H^1$-regularity preserving, but corresponds to the physical meaningful interpretation that there is no fluid flow in $\Oes$. We also extend $\nabla\hve$ by $0$ to $\Omega$, which we denote by $\widetilde{\nabla \hve}$
Since $\hvGe$ is of order $\e$, it vanishes in the limit $\e \to 0$ and $\hve$ and $\hwe$ have the same two-scale limits. We denote the two-scale limit of $\hve$ and $\hwe$ by $\hvn$ because $\hve$ corresponds with the physical meaningful quantity.

In a first step, we show that $\hve$ and some extension of the pressure $\hpe = \hqe + \hpe$ two-scale converge to solutions of the two-scale limit system, which is given by the following instationary two-pressure Stokes equation:
\begin{subequations}\label{eq:Strong:two-pressure:Trafo}
	\begin{align}\notag
		\partial_t (\Anm \hvn) - \nabla_y^\top (\Anm \hvn) \Pnm \partial_t \psin	-\Jnm \mu \div_y(&\Anm \PnmT\nabla_y(\Anm \hvn)
	\\\label{eq:Strong:two-pressure:Trafo:1}
			+\PnmT\nabla_x \hp + \PnmT\nabla_y \hpO &= \hat{f} && \textrm{in } (0,T) \times \Omega\times \Yp \, ,
		\\
		\label{eq:Strong:two-pressure:Trafo:2}
		\Jnm\div_y (\hvn) &= 0 &&\textrm{in } (0,T) \times \Omega\times \Yp \, ,
		\\\label{eq:Strong:two-pressure:Trafo:3}
\div_x \left( \int_{\Yp} \hvn \dy \right) = - \int_{\Yp} \div_y (&\hvGe) \dy &&\textrm{in } (0,T) \times \Omega \, ,
		\\\label{eq:Strong:two-pressure:Trafo:4}
		\hvn &= 0 &&\textrm{on } (0,T) \times \Omega\times \Gamma \, ,
		\\\label{eq:Strong:two-pressure:Trafo:5}
		\hp &= \hqb&&\textrm{on } (0,T) \times \Omega \, ,
		\\\label{eq:Strong:two-pressure:Trafo:6}
	y &\mapsto \hvn, \, \hpO &&Y\textrm{-periodic} \, ,
		\\\label{eq:Strong:two-pressure:Trafo:7}
		\hvn &= \hvn^\init && \textrm{in } \Omega\times \Yp\, .
	\end{align}
\end{subequations}

In order to formulate the weak form of \eqref{eq:Strong:two-pressure:Trafo}, one can proceed as in the $\e$-scaled case. One has to multiply \eqref{eq:Strong:two-pressure:Trafo:1} by $\PnT$ and to employ the product rule for the time-derivative term.
The weak form is given by:
 Find $(\hvn, \hq, \hqO) \in L^2((0,T) \times \Omega;H^1_{\Gamma \#}(\Yp;\R^d))) \times L^2(0,T;H^1_0(\Omega)) \times L^2((0,T) \times \Omega; L^2_0(\Yp))$ with $\partial_t \hvn, \partial_t (\PnT \Anm \hvn) \in L^2((0,T) \times \Omega \times \Yp;\R^d)$
such that
\begin{subequations}\label{eq:WeakForm:TwoPressure:Trafo}
	\begin{align}
		\begin{aligned}\label{eq:WeakForm:TwoPressure:Trafo:1}
			&\intTOYp \partial_t (\PnT \Anm \hvn) \cdot \varphi - \left( \partial_t \PnT \Anm \hvn -\PnT \nabla^\top (\Anm\hvn)\Pnm \partial_t \psin \right) \cdot \varphi \dyxt
			\\
			&+
			\intTOYp \mu \An \PnmT \nabla_y (\Anm\hvn) : \nabla_y (\Anm\varphi) + \nabla_x \hq \cdot \varphi -\hqO \div_y(\varphi) \dyxt 
			\\
			&=
			\intTOYp \PnT \hat{f} \cdot \varphi - (\nabla_x \hpb + \nabla_y \hpbO ) \cdot \varphi \dxt \, ,
		\end{aligned}
		\\
		\begin{aligned}\label{eq:WeakForm:TwoPressure:Trafo:2}
			\intTOYp \phi_1 \div_y(\hvn) \dxt &= 0 \, ,
		\end{aligned}
		\\
		\begin{aligned}\label{eq:WeakForm:TwoPressure:Trafo:3}
			\intTO \phi \div_x\left(\intYp\hvn \dy \right) \dxt &= - \intTOYp \phi \div_y(\hvGn) \dxt \, ,
		\end{aligned}
		\\
		\begin{aligned}\label{eq:WeakForm:TwoPressure:Trafo:4}
			\hvn(0) &= \hvn^\init
		\end{aligned}
	\end{align}
\end{subequations}
for all $(\varphi, \phi, \phi_1) \in L^2((0,T) \times \Omega;H^1_\Ge(\Yp;\R^d)) \times L^2((0,T) \times \Omega) \times L^2((0,T) \times \Omega \times \Yp)$.

\begin{thm}[Convergence result for the solutions of the substitute problem]\label{thm:TwoPressureStokes:Trafo}
Let $(\hwe, \hqe)$ be the solution of \eqref{eq:WeakForm:Eps:Trafo} and $\widetilde{\hwe}$ and $\hQe$ their extensions as defined above. Then, 
\begin{align*}
\hwe \tsw{2} \hvn \, , \qquad
\hQe \tsw{2} \hq \, ,
\end{align*}
where $(\hvn, \hq, \hqO) \in L^2((0,T) \times \Omega;H^1_{\Gamma \#}(\Yp;\R^d))) \times L^2(0,T;H^1_0(\Omega)) \times L^2((0,T) \times \Omega; L^2_0(\Yp))$
are the unique solution of the instationary two-pressure Stokes equations \eqref{eq:WeakForm:TwoPressure:Trafo}.
\end{thm}
\begin{proof}
Since $\widetilde{\hwe}$ and $\e \widetilde{\nabla \hwe}, \partial_t \hwe$ are bounded, standard two-scale compactness arguments provide a subsequence and $\hvn \in L^2((0,T) \times \Omega;H^1_\#(Y;\R^d))$ with $\partial_t \hvn \in L^2((0,T) \times \Omega \times Y;\R^d)$ such that $\widetilde{\hwe}$, $\e \widetilde{\nabla \hwe}$ and $\partial_t \hwe$ two-scale converge weakly to $\hvn$, $\nabla_y \hvn$ and $\partial_t\hvn$, respectively, where $\hvn$ is zero on $Y \setminus \Yp$ and, thus, can be identified with an element in $L^2((0,T) \times \Omega;H^1_{\Gamma \#}(\Yp;\R^d))$. With the two-scale convergence of $\hwe(0)$ to $\hvn^\init$, we get $\hvn(0) = \hvn^\init$.
Testing the divergence condition \eqref{eq:WeakForm:Eps:Trafo:2} 
with $\phi\left(t,x,\tfrac{x}{\e} \right) $ for  $\phi\in C^\infty([0,T];C^\infty_c(\Omega; C^\infty_\#(Y)))$ yields the microscopic incompressibility condition \eqref{eq:WeakForm:TwoPressure:Trafo:2}.
Testing the divergence condition \eqref{eq:WeakForm:Eps:Trafo:2} 
with $\phi \in C^\infty([0,T];C^\infty_c(\Omega))$ yields the macroscopic divergence condition \eqref{eq:WeakForm:TwoPressure:Trafo:3}.

Using the boundedness of $\hQe$, we can pass to a further subsequence and get $\hQ \in L^2((0,T) \times \Omega \times Y)$ such that $\hQe$ two-scale converges to $\hq$.
In order to show that $\hq$ is constant on $Y$, we test \eqref{eq:WeakForm:Eps:Trafo:1} by $\e \varphi(t,x,\tfrac{x}{\e})$ for $\varphi \in C^\infty((0,T) \times \Omega; C^\infty_\#(Y))$. Due to this $\e$-factor, all the terms converge to $0$ besides the pressure term and, thus, we get
\begin{align*}
0 &= \lim\limits_{\e \to 0} \intTOe \hqe\tx \div_y \left(\varphi \left(t,x,\tfrac{x}{\e}\right) \right) \dxt 
\\
&= \intTOYp \hq \txy \div_y \left(\varphi (t,x,y) \right) \dxt
\end{align*} 
and, consequently, $\nabla_y \hq = 0$ on$ \Yp$. By the construction of the extension $\hQe$ one can deduce further $\nabla_y \hq = 0$ on $Y$ and, thus, $\hq \in L^2((0,T) \times \Omega)$. 

We test \eqref{eq:WeakForm:Eps:Trafo:1} by $\varphi(t,x,\tfrac{x}{\e})$ for $\varphi \in C^\infty((0,T) \times \Omega;H^1_{Y \#}(\Yp;\R^d))$ with $\div_y(\varphi) = 0$. To pass to the limit $\e \to 0$, we employ the two-scale convergences of the unknowns $\hwe$ and $\hqe$ as well as of the coefficient and data given by Lemma \ref{lem:TwoScaleJacobians} and \eqref{eq:Two-scale-conv:data:Trafo}, respectively. Moreover, we note that $\hvGe$, $\partial_t \hvGe$ and $\e \nabla \hvGe$ are of order $\e$ and, thus, the terms with them vanish in the limit of \eqref{eq:WeakForm:Eps:Trafo:1} and we get
\begin{align}
		\begin{aligned}\label{eq:WeakForm:TwoPressure:Trafo:1:Dist:hQ}
	&\intTOYp \partial_t (\PnT \Anm \hvn) \cdot \varphi - \left( \partial_t \PnT \Anm \hvn -\PnT \nabla^\top (\Anm\hvn)\Pnm \partial_t \psin \right) \cdot \varphi 
	\\
	&\qquad \qquad +
 \mu \An (\PnmT \nabla_y (\Anm\hvn) + \PnmT \nabla_y (\Anm\hvn)^\top) : \nabla_y (\Anm\varphi) 
	\\
	& \qquad\qquad -
	 \hq \div_x(\varphi)  \dyxt 
\\
&	=
	\intTOYp \PnT \hat{f} \cdot \varphi - (\nabla_x \hpb + \nabla_y \hpbO) \cdot \varphi \dxt \, .
\end{aligned}
\end{align}
We test \eqref{eq:WeakForm:TwoPressure:Trafo:1:Dist:hQ} with $\varphi_0 \varphi_i$ for $\varphi_0 \in C^\infty((0,T)\times \Omega)$ and $\varphi_i \in H^1_{\Gamma \#}(\Yp;\R^d)$ such that $\intYp \varphi_1 \dy = e_i$ for $i \in \{1, \dots, d\}$. Such functions $\varphi_i$ can be constructed similarly to the proof of \cite[Lemma 2.10]{All92}.
For these test function, we can rewrite the pressure term as $\intTOYp \hq \div_x(\varphi_0 \varphi_1) \dyxt = \intTO \hq \partial_{x_i}\varphi_0 \dyxt$, while we interpret the remaining terms as functional for $\varphi \in L^2((0,T) \times \Omega)$. Consequently, $\hq \in L^2(0,T;H^1(\Omega))$ and we can integrate the macroscopic pressure term in \eqref{eq:WeakForm:TwoPressure:Trafo:1:Dist:hQ} by parts.
By a density argument, the resulting equation holds for all $\varphi \in L^2((0,T) \times \Omega;H^1_{Y \#}(\Yp;\R^d))$ with $\div_y(\varphi) = 0$. 

In order to satisfy the equation for all test functions $\varphi \in L^2((0,T) \times \Omega;H^1_{\Gamma \#}(\Yp;\R^d))$, we add a microscopic pressure $\hqO$. For this, we note that the Bogovskii operator provides the surjectivity of $\div_y$ from $L^2((0,T) \times \Omega;L^2_0(\Yp))$ onto $L^2((0,T) \times \Omega;H^1_{\Gamma \#}\Yp;\R^d))$. Consequently, $\div_y$ has closed range and the closed range theorem provides $\hqO$ such that 
\begin{align}
	\begin{aligned}\label{eq:WeakForm:TwoPressure:Trafo:1:Symmetrised}
		&\intTOYp \partial_t (\PnT \Anm \hvn) \cdot \varphi - \left( \partial_t \PnT \Anm \hvn -\PnT \nabla^\top (\Anm\hvn)\Pnm \partial_t \psin \right) \cdot \varphi 
		\\
		&
		\qquad \qquad +\mu \An (\PnmT \nabla_y (\Anm\hvn) + \PnmT \nabla_y (\Anm\hvn)^\top) : \nabla_y (\Anm\varphi) 
		\\
		&\qquad \qquad +
		 \nabla_x \hq \cdot \varphi - \hqO \div_y(\varphi) \dyxt 
		\\
		&=
		\intTOYp \PnT \hat{f} \cdot \varphi - (\nabla_x \hpb + \nabla_y \hpbO )\cdot \varphi \dxt
	\end{aligned}
\end{align}
holds for all test functions $\varphi \in L^2((0,T) \times \Omega;H^1_{\Gamma \#}(\Yp;\R^d))$.

To deduce \eqref{eq:WeakForm:TwoPressure:Trafo:1}, it remains to show that $$\intYp \mu \An (\PnmT \nabla_y (\Anm\hvn) + \PnmT \nabla_y (\Anm\hvn)^\top) : \nabla_y (\Anm\varphi) \dy = 0$$ for a.e.~$\tx \in \Omega \times \Yp$. This can be done following the argumentation in the end of the proof of \cite[Theorem 4.7]{JDE24}.

Since this argumentation holds also after passing to an arbitrary subsequence before and the fact that solution of \eqref{eq:WeakForm:TwoPressure:Trafo} is unique (see Lemma \ref{lem:Existence:TwoPressureStokes} below), the convergence holds for the whole sequence.
\end{proof}

\begin{lemma}[Existence and uniqueness of the solution of the two-pressure Stokes equations] \label{lem:Existence:TwoPressureStokes}
There exists a unique solution $(\hvn, \hq, \hqO) \in L^2((0,T) \times \Omega;H^1_{\Gamma \#}(\Yp;\R^d))) \times L^2(0,T;H^1_0(\Omega)) \times L^2((0,T) \times \Omega; L^2_0(\Yp))$ with $\partial_t \hvn, \partial_t (\PnT \Anm \hvn) \in L^2((0,T) \times \Omega \times \Yp)$ of \eqref{eq:WeakForm:TwoPressure:Trafo}.
\end{lemma}
\begin{proof}
Indeed, the existence of the solution is already secured by the homogenisation process. In order to show the uniqueness, one can reformulate \eqref{eq:WeakForm:TwoPressure:Trafo} in the abstract setting of Proposition \ref{prop:Generic:Existence} similarly as in the $\e$-scaled case. The construction and the estimates of $\mathcal{M}$ and $\mathcal{A}$ corresponds to the $\e$-scaled case and become even simpler since there is no Korn-type inequality required. We define the operator $\mathcal{B} \in \Lin(L^2(\Omega;H^1_{\Gamma \#}(\Yp;\R^d), (H^1(\Omega)\times L^2(\Omega;L^2_0(\Yp)) )^*)$ by 
\begin{align*}
	\mathcal{B}(v, (p,p_1)) \coloneqq \intOYp \nabla p_1 \cdot \varphi - p_1 \div_y(v)) \dyx 
\, .
\end{align*}
Its inf--sup stability can be shown as in Lemma \cite[Lemma 4.10]{JDE24}.

For the compatibility of the initial values $\hvn$, we note that
$\hvn \in L^2((0,T) \times \Omega;H^1_{\Gamma \#}(\Yp;\R^d)))$. Moreover, one can show $\mathcal{B}(\hvn) = \int\limits_Y\div_y(\hvGn) \dy$ arguing as is the derivation of \eqref{eq:WeakForm:TwoPressure:Trafo:2}--\eqref{eq:WeakForm:TwoPressure:Trafo:3} by employing the compatibility of the $\e$-scaled initial values.
\end{proof}

\section{A Darcy law with memory for evolving microstructure}\label{sec:DarcySubst}
In \eqref{eq:WeakForm:TwoPressure:Trafo}, the macroscopic pressure term $\nabla_x \hp$ contributes as source term similarly to $\PnT\hat{f}$. These two terms differ in their microscopic structure, i.e.~$\nabla_x \hp$ is independent of $y$ while $\PnT\hat{f}$ has the $y$-dependence arising from $\PnT$. Thus, one would have to construct two different cell problems in order to account for this coefficient.
The following computation shows that the difference between the source terms $\nabla_x \hat{q}$ and $\PnT \nabla_x \hat{q}$ leads only to an additive difference of the microscopic pressure term.
Let $\varphi \in H^1_{\Gamma \#} (\Yp, \R^d)$, then we get by integrating by parts
\begin{align*}
	\intYp \PnT e_i \cdot \varphi \dy &= \intYp (\nabla_y (\psin-y) + \1) e_i \cdot \varphi \dy =
	\intYp (\nabla_y (\psin-y)_i + \1) e_i \cdot \varphi \dy \\
	& =
	\intYp e_i \cdot \varphi \dy - \intYp (\psin-y)_i \div_y(\varphi) \dy \, ,
\end{align*}
where the boundary term of the integration by parts vanishes since $\varphi$ is zero on $\Gamma$ and $\varphi$ and $\psin\txy -y$ are $Y$-periodic.
This computation allows us to rewrite the macroscopic pressure terms in \eqref{eq:WeakForm:TwoPressure:Trafo:1} leading to
\begin{align}
	\begin{aligned}\label{eq:WeakForm:TwoPressure:Trafo:1:PressureAdjusted}
		&\intTOYp \partial_t (\PnT \Anm \hvn) \cdot \varphi - \left( \partial_t \PnT \Anm \hvn -\PnT \nabla^\top (\Anm\hvn)\Pnm \partial_t \psin \right) \cdot \varphi 
		\\
		&\qquad \qquad +
		 \mu \An \PnmT \nabla_y (\Anm\hvn) : \nabla_y (\Anm\varphi) + \PnT\nabla_x \hq \cdot \varphi -\hqO' \div_y(\varphi) \dyxt 
		\\
		&=
		\intTOYp \PnT \big(\hat{f} - (\nabla_x \hpb + \nabla_y \hpbO' )\big) \cdot \varphi \dxt \, ,
	\end{aligned}
\end{align}
for
\begin{align*}
	\hqO' = \hqO - (\psin-y) \cdot \nabla_x \hq \, \qquad \hpbO' = \hpbO - (\psin-y) \cdot \nabla_x \hq
\end{align*}
instead of \eqref{eq:WeakForm:TwoPressure:Trafo:1}. 
The same substitution but in the strong form \eqref{eq:Strong:two-pressure:Trafo:1}, i.e.~
\begin{align*}
\hpO' = \hpO - (\psin-y) \cdot \nabla_x \hp \, ,
\end{align*}
cancels the coefficient $\PnmT$ in front of the macroscopic pressure $\nabla_x \hp$, i.e.~one can replace \eqref{eq:Strong:two-pressure:Trafo:1} by 
\begin{align}
	\begin{aligned}
		\partial_t (\Anm \hvn) - \nabla_y^\top (\Anm \hvn) \Pnm \partial_t \psin	-\Jnm \mu \div_y(&\Anm \PnmT\nabla_y(\Anm \hvn)
		\\\label{eq:Strong:two-pressure:Trafo:1:PressureAdjusted}
		+\nabla_x \hp + \PnmT\nabla_y \hpO' &= \hat{f} && \hspace{-1cm} \textrm{in } (0,T) \times \Omega\times \Yp \, .
	\end{aligned}
\end{align}

Having this reformulation, it suffices to consider only one type of cell problem for the contribution of the macroscopic bulk term. The cell problem and its solutions $(\hzi, \hat{\pi}_i)$ are parametrised over the initial times $s \in (0, T)$ and the macroscopic position $x \in \Omega$ and the direction $e_i$ for the initial values with $i \in \{1 , \dots, d\}$. It is given by:
For given $(s,x)\in (0,T) \times \Omega$, find $\hzi(s,x,t,y)$ and $\hat{\pi}_i(s,x,t,y)$ such that
\begin{subequations}\label{eq:Strong:two-pressure:Trafo:CellProb:A}
	\begin{align}\notag
		\partial_t (\Anm \hzi) - \nabla_y^\top (\Anm \hzi) \Pnm \partial_t \psin&
		\\	-\Jnm  \div_y(\Anm \PnmT\nabla_y(\Anm \hzi)
		\label{eq:Strong:two-pressure:Trafo:CellProb:A:1}
		+ \nabla_y \hat{\pi}_i &= 0 && \textrm{in } (s,T) \times \Yp \, ,
		\\
		\label{eq:Strong:two-pressure:Trafo:CellProb:A:2}
		\Jnm\div_y (\hzi) &= 0 &&\textrm{in } (s,T) \times \Yp \, ,
		\\
		\hzi &= 0 &&\textrm{on } (s,T) \times \Gamma \, ,
		\\
		y \mapsto \hzi, & \hpO &&Y\textrm{-periodic} \, ,
		\\
		\Anm(s)\hzi(s) &= e_i && \textrm{in } \Yp\, .
	\end{align}
\end{subequations}

The second cell problem and its solution $(\hat{\zeta}^\init, \hat{\pi}^\init)$ accounts for the contribution of the initial value of \eqref{eq:Strong:two-pressure:Trafo}. It is parametrised over the macroscopic position $x \in \Omega$ and is given by 
\begin{subequations}\label{eq:Strong:two-pressure:Trafo:CellProb:B}
	\begin{align}\notag
		\partial_t (\Anm \hat{\zeta}^\init) - \nabla_y^\top (\Anm \hat{\zeta}^\init) \Pnm \partial_t \psin&
		\\	-\Jnm \mu \div_y(\Anm \PnmT\nabla_y(\Anm \hat{\zeta}^\init)
		\label{eq:Strong:two-pressure:Trafo:CellProb:B:1}
		+ \PnmT\nabla_y \hat{\pi}_0 &= 0 && \textrm{in } (0,T) \times \Yp \, ,
		\\
		\label{eq:Strong:two-pressure:Trafo:CellProb:B:2}
		\Jnm\div_y (\hat{\zeta}^\init) &= 0 &&\textrm{in } (0,T) \times \Yp \, ,
		\\
		\hat{\zeta}^\init &= 0 &&\textrm{on } (0,T) \times \Gamma \, ,
		\\
		y \mapsto \hat{\zeta}^\init, & \hpO &&Y\textrm{-periodic} \, ,
		\\
		\Anm(0)\hat{\zeta}^\init(0) &= \hvn^\init && \textrm{in } \Yp\, .
	\end{align}
\end{subequations}

Comparing these two-cell problems with \eqref{eq:Strong:two-pressure:Trafo:1:PressureAdjusted}, \eqref{eq:Strong:two-pressure:Trafo:2}--\eqref{eq:Strong:two-pressure:Trafo:7} leads to
\begin{align}\label{eq:hvn=hzeta}
	\hvn\txy &= \hat{\zeta}^\init\txy + \frac{1}{\mu}\int\limits_0^t \sum\limits_{i=1}^d \hzi(s,x,t,y) \big( \hat{f}_i - \partial_{x_i} \hp \big) \, ,
	\\\label{eq:hpO=hpi}
	\hqO'\txy + \hpbO'\txy &= \hat{\pi}_0\txy +  \frac{1}{\mu}\int\limits_0^t \sum\limits_{i=1}^d \hat{\pi}_i (s,x;t,y) \big(\hat{f}_i - \partial_{x_i} \hp \big) \, .
\end{align}

\begin{rem}
Note that the bulk source terms in \eqref{eq:Strong:two-pressure:Trafo:1:PressureAdjusted} becomes an initial value in the cell problem \eqref{eq:Strong:two-pressure:Trafo:CellProb:A}. Thus, the solution of the cell problem does not have the same physical unit, namely $\hzi$ are accelerations of the fluid and not velocities. This is also addressed in \eqref{eq:hvn=hzeta} by the time integration.
This reformulation leads to incompatible initial data in \eqref{eq:Strong:two-pressure:Trafo:CellProb:A}, i.e.~the initial values do not satisfy the zero boundary values at $\Gamma$. Thus, it does not satisfy the assumptions of Proposition \ref{prop:Generic:Existence}. Therefore, we have to use a weaker solution concept, which provides the time derivative and the pressure only in a distributional sense. For this, one can reformulate \eqref{eq:Strong:two-pressure:Trafo:CellProb:A} in an operator formulation for which \cite[Theorem 7.14]{Zim21} provides the existence and \cite[Theorem 7.19]{Zim21} the uniqueness of a solution. 

For the case of a stationary domain one can integrate the first cell problem over the time and consider there an inhomogeneous right-hand side in the momentum equation and a homogeneous initial value (see \cite{All92a}). Then, one has to integrate over the time derivative of the cell problem in order to identify the two-scale limit of the velocity.
In our case of a moving domain, this integration of the cell problems would lead to additional terms due to the time-dependent coefficients and thus complicate the structure. 
\end{rem}

In order to compute the effective fluid velocity, we define
\begin{align*}
\hv \coloneqq \intYp \hvn \dy \, .
\end{align*}
Separating the macroscopic and microscopic variable in \eqref{eq:hvn=hzeta} and in \eqref{eq:Strong:two-pressure:Trafo:3} gives the following integro-differential equation as homogenised system
\begin{subequations}\label{eq:DarcyMemory:Trafo}
\begin{align}\label{eq:DarcyMemory:Trafo:}
\hv \tx &= \hv^\init + \int\limits_{0}^t \hat{K}(s,t,x) (\hat{f} - \nabla \hp)(s,x) \ds &&\textrm{in } (0,T) \times \Omega \, ,
\\\label{eq:DarcyMemory:Trafo:2}
\div_x ( \hv\tx) &= - \int_{\Yp} \div_y (\hvGn)\txy \dy &&\textrm{in } (0,T) \times \Omega \, ,
\\\label{eq:DarcyMemory:Trafo:3}
\hp\tx &= \hpb\tx &&\textrm{on } (0,T) \times \Omega \, ,
\end{align}
where the permeability-type tensor $\hat{K}(s,t,x)$ is given by
\begin{align*}
\hat{K}(s,t,x)_{ji} \coloneqq \intYp \hzi(s,t,x,y) \cdot e_j\dy 
\end{align*}
\end{subequations}
and the initial value $\hv^\init$ by
\begin{align*}
\hv^\init\tx = \intYp \hat{\zeta}_0\txy \dy \, .
\end{align*}

In order to give $\hv$ a physical interpretation, we consider the back-transformation of the limit equations in the following section.

In the case that we model a no-slip boundary condition at $\vGe$, one has $\vGe\tx = \partial_t \psie(t,\psi_\e^{-1}\tx)$, which gives $\hvGe = \Ae \partial_t \psi$ and allows the following simplification.
\begin{lemma}[Macroscopic divergence condition for a microscopic no-slip boundary condition]\label{lem:NoSlipSimplification}
Assume that $\vGe\tx = \partial_t \psie(t,\psi_\e^{-1}\tx)$ for $x \in \Ge(t)$. 
Then, one can simplify the right-hand side of the macroscopic divergence condition \eqref{eq:WeakForm:TwoPressure:Trafo:3} into
	\begin{align}\label{eq:WeakForm:TwoPressure:Trafo:3:NoSlip}
		\begin{aligned}
			\intTO \phi \div_x\left(\intYp\hvn \dy \right) \dxt &= - \intTOYp \phi \div_y(\hvGn) \dxt \\
			&= - \intTO \phi\frac{\dd}{\dt} \Theta \dxt
		\end{aligned}
\end{align}
for $\Theta \tx = |\Yptx|$\,.
\end{lemma}
\begin{proof}
The Jacobi formula states that almost everywhere
\begin{align*}
	\frac{\dd}{\dt} \det (A(t)) 
	=\operatorname{tr}(\operatorname{Adj}(A(t)) \partial_t A(t))
	=\det(A(t)) A^{-1}(t) : \partial_t A^\top(t)
\end{align*}
for every $A \in W^{1,\infty}(0,T)^{n \times n}$.
With the Leibniz rule, the Jacobi formula applied to $\partial_y \psin$ and the Piola identity, we infer 
\begin{align*}
	\div_y&(\An\txy \partial_t \psin\txy) 
	\\&= \An\txy : \nabla \partial_t \psin\txy + \div_y(\An\txy)\partial_t \psin\txy 
	\\
	&= \partial_t \Jn\txy + 0 \partial_t \psin\txy = \partial_t \Jn\txy \,.
\end{align*}
Hence, we obtain
\begin{align*}
	\intYp \div_y (\An\txy \hvGe\txy) \dy &= \intYp \partial_t \Jn\txy \dy 
	=\partial_t \intYp\Jn\txy \dy 
	\\
	&= \frac{\dd}{\dt} \Theta\tx.
\end{align*} 
\end{proof}
As consequence of Lemma \ref{lem:NoSlipSimplification}, we can simplify the right-hand side of \eqref{eq:DarcyMemory:Trafo:2} by
\begin{align}\label{eq:DarcyMemory:Trafo:2:NoSlip}
\div_x ( \hv\tx) &= -\frac{\dd}{\dt} \Theta\tx &&\textrm{in } (0,T) \times \Omega \,.
\end{align}

\section{Back-transformation of the two-pressure Stokes system}\label{sec:BackTrafo}
In the last step, we transform the two-pressure Stokes equations \eqref{eq:Strong:two-pressure:Trafo}, the cell problems and the cell problems \eqref{eq:Strong:two-pressure:Trafo:CellProb:A} and \eqref{eq:Strong:two-pressure:Trafo:CellProb:A} back to the actual moving cell domains.
We separate again the macro- and microscopic variable, which leads to the Darcy law with memory \eqref{eq:Strong:DarcyMemory}. We note that a-priori it is even from a formal point of view not clear that $\hv = v=\intYptx\vn\dy$, due to the transformation coefficient $\Pnm$ in $\hvn = \An\vn = \Jn \Pnm \vn$. Nevertheless, we can employ the microscopic incompressibility condition in order to identify $\hv$ with $v$. Moreover, we show that $\hat{K} = K$.

Let
\begin{align*}
\vn\txy &\coloneqq \Anm(t, x,\psinm\txy) \hvn(t,x, \psinm\txy) \, , \\
\hvn\txy & = \An(t, x,y) \vn(t,x, \psin\txy) \, .
\end{align*}
This choice is motivated by Lemma \ref{lem:TwoScaleEquiv}, since it gives
\begin{align}\label{eq:EquivalenceConvergence}
\ve \tsw{2,2} \vn \textrm{ if and only if } \hve \tsw{2,2} \hvn \, .
\end{align}

We use this substitution in the two-pressure Stokes equations \eqref{eq:Strong:two-pressure:Trafo} with the version \eqref{eq:Strong:two-pressure:Trafo:1:PressureAdjusted} for the momentum equation. 
The resulting equation can be transformed to the moving domain by means of $\psin$ similarly to the transformation of the $\e$-scaled Stokes equations leading to \eqref{eq:Strong:two-pressure}.
However, one has to be careful with the transformation of the pressure terms and the macroscopic divergence condition. 
The macroscopic pressure term in \eqref{eq:Strong:two-pressure:Trafo} has the coefficient $\PnmT$, which does not cancel in the back-transformation since there is no $y$-derivate.
This can be circumvented by the substitution of \eqref{eq:Strong:two-pressure:Trafo} by \eqref{eq:Strong:two-pressure:Trafo:1:PressureAdjusted}, where this coefficient is cancelled by the substitution of the microscopic pressure $\hqO'$. For the weak form, one has to use analogously \eqref{eq:WeakForm:TwoPressure:Trafo:1:PressureAdjusted} instead of \eqref{eq:WeakForm:TwoPressure:Trafo:1}.

A similar problem arises in the back-transformation of the left-hand side of the macroscopic divergence condition, where the factor $\Pnm$ of $\An = \Jn \Pnm$ does not cancel since there is no $y$-derivative involved, i.e.~
\begin{align*}
\intYp \hvn \dy &= \intYp \An(t, x,y) \vn(t,x, \psin\txy) \dy
\\
&= \intYptx \Pnm(t, x, \psinm\txy) \vn\txy \dy \,.
\end{align*}
As for the pressure, we separate the microscopic oscillating part of the coefficient. But instead of shifting it to a microscopic term, we show that it cancels in this term due to the microscopic incompressibility of $\hvn$, i.e.~we rewrite $\hvn \txy = \An(t, x,y) \vn(t,x, \psin\txy)$ and use Lemma~\ref{lem:J0=A0}, which is given below, for $u\txy = \vn(t,x, \psin\txy)$ to deduce
\begin{align}
	\begin{aligned}\label{eq:hvn=vn}
	\intYp \hvn \dy &= \intYp \An(t,x,y) \vn(t,x, \psin\txy) \dy
\\&= \intYp \Jn(t,x,y) \vn(t,x, \psin\txy) \dy
= \intYptx \vn \dy \,.
	\end{aligned}
\end{align}

\begin{lemma}\label{lem:J0=A0} 
Let $u \in L^2(\Omega;H^1_{\Gamma\#}(\Yp;\R^d))$ with
\begin{align}\label{eq:MicroIncompressibility:RefrenceCoord}
		\div_y(\An\xy u\xy) \dy = 0
\end{align}
for a.e.~$x \in \Omega$.
Then,
\begin{align}\label{eq:intYA=intYJ}
\intYp \An\xy u\xy \dy = \intYp \Jn\xy u\xy \dy
\end{align}
for a.e.~$x \in \Omega$.
\end{lemma}
\begin{proof}
For $\xi \in \R^d$, we note that
\begin{align}
	\begin{aligned}\label{eq:An:hwn=Jn:hwn+nabla}
		\An\xi &= \Jn \Pnm \xi
		=
		\Jn \xi + ( \1 -\Pn) \Jn\Pnm \xi
		= 
		\Jn \xi + \partial_y (y -\psin) \An \xi
		\\
		&=
		\Jn \xi + \left(\begin{array}{c}
			\nabla_y ((y -\psin)_1) \cdot \An \xi
			\\
			\vdots
			\\
			\nabla_y ((y -\psin)_d) \cdot \An \xi
		\end{array}\right).
	\end{aligned}
\end{align}
We set $\xi= u$ for $u \in L^2(\Omega;H^1_{\Gamma\#}(\Yp;\R^d))$ with $\div_y(\An\xy u\xy)= 0$. Then, we integrate the second summand on the right-hand side of \eqref{eq:An:hwn=Jn:hwn+nabla} over $\Yp$, subsequently integrate by parts and use the microscopic incompressibility condition \eqref{eq:MicroIncompressibility:RefrenceCoord}. This shows
\begin{align*}
	\begin{aligned}
		\intYp& 
		\nabla_y (y_i -\psin\xy_i) \cdot \An\xy u\xy\dy
		\\
		&=
		-\intYp 
		(y_i -\psin\xy_i) \cdot \div_y(\An\xy u\xy) \dy
		= 0
	\end{aligned}
\end{align*}
for every $i \in \{1,\dots, d\}$,
where the boundary integral of the integration by parts vanishes on $\Gamma$ since $\hvn$ is zero and vanishes on $\partial Y \cap \partial \Yp$ since $y-\psin$, $\An$ and $u$ are $Y$-periodic.
Therefore, the second summand on the right-hand side of \eqref{eq:An:hwn=Jn:hwn+nabla} has mean value zero and vanishes after integrating over $\Yp$, which yields \eqref{eq:intYA=intYJ}.
\end{proof}

The strong form of the back-transformed two-pressure Stokes equations is given by
\begin{subequations}\label{eq:Strong:two-pressure}
\begin{align}
\partial_t \vn - \mu \Delta_{yy}(\vn) +\nabla_x p + \nabla_y p_1 &= f && \textrm{for } \tx \in (0,T) \times \Omega \, , y\in \Yptx \, ,
\\
\div_y (\vn) &= 0 &&\textrm{for } \tx \in (0,T) \times \Omega \, , y\in \Yptx \, ,
\\
\div_x \left( \int_{\Yptx} \vn \dy \right) = - \int_{\Yptx} \div_y (&\vGe) \dy && \textrm{for } \tx \in (0,T) \times \Omega \, ,
\\
\vn &= 0 && \textrm{for } \tx \in (0,T) \times \Omega \, , y\in \Gtx \, ,
\\
p &= \pb &&\textrm{on } (0,T) \times \Omega \, ,
\\
y &\mapsto\vn, \, p_1 &&Y\textrm{-periodic} \, ,
\\
\vn &= \vn^\init &&\textrm{for } x \in \Omega \, , y\in \Yp(0,x) \, .
\end{align}
\end{subequations}
Similarly, one can transform-back the weak form of the two-pressure Stokes equations, which leads to:
Find $(\vn, q, \qO) \in L^2((0,T) \times H^1_{\#\Gtx}(\Yptx; \R^d)) \times L^2(0,T;H^1_0(\Omega)) \times L^2((0,T) \times \Omega;L^2_0(\Yptx))$ with $\partial_t \vn \in L^2(\{\txy \in (0,T) \times \Omega \times Y\mid y\in \Yptx \};\R^d)$ such that
\begin{subequations}\label{eq:WeakForm:TwoPressure}
	\begin{align}
		\begin{aligned}\label{eq:WeakForm:TwoPressure:1}
			&\intTOYptx \partial_t \vn \cdot \varphi +
			 \mu \nabla_y \vn : \nabla_y \varphi + \nabla_x q \cdot \varphi -\qO \div_y(\varphi) \dyxt 
			\\
			&=
			\intTOYptx f \cdot \varphi - (\nabla_x \pb + \nabla_y \pbO ) \cdot \varphi \dxt \, ,
		\end{aligned}
		\\
		\begin{aligned}\label{eq:WeakForm:TwoPressure:2}
			\intTOYptx \phi_1 \div_y(\vn) \dxt &= 0 \, ,
		\end{aligned}
		\\
		\begin{aligned}\label{eq:WeakForm:TwoPressure:3}
			-\intTO \nabla_x \phi \intYptx\vn \dy \dxt &= - \intTOYp \phi \div_y(\vGn) \dxt \, ,
		\end{aligned}
		\\
		\begin{aligned}\label{eq:WeakForm:TwoPressure:4}
			\vn(0) &= \vn^\init
		\end{aligned}
	\end{align}
\end{subequations}
for all $(\varphi, \phi, \phi_1) \in H^1_{\Gtx}(\Yptx; \R^d)) \times L^2(0,T;H^1_0(\Omega)) \times L^2((0,T) \times \Omega;L^2_0(\Yptx))$.

The space $L^2((0,T) \times H^1_{\#\Gtx}(\Yptx; \R^d))$ has to be understood as the subspace of $L^2((0,T) \times H^1_{\#}(Y; \R^d))$ of functions that are $0$ in $\Ystx$ for a.e.~$\tx \in(0,T) \times \Omega$.
We understand $\partial_t \vn \in L^2(\{\txy \in (0,T) \times \Omega \times Y\mid y\in \Yptx \};\R^d)$ in the sense that $\partial_t \vn \in L^2((0,T)\times \Omega \times Y;\R^d)$ and $\partial_t \vn = 0$ in $\Ystx$ for a.e.~$\tx \in(0,T) \times \Omega$.

We note that \eqref{eq:WeakForm:TwoPressure:3} shows that $\div_x(\intYptx\vn \dy) \in L^2((0,T) \times \Omega)$.

By transforming the weak forms, one obtains the equivalence of the weak form \eqref{eq:WeakForm:TwoPressure} to the weak form \eqref{eq:WeakForm:TwoPressure:Trafo:2}--\eqref{eq:WeakForm:TwoPressure:Trafo:4} with \eqref{eq:WeakForm:TwoPressure:Trafo:1:PressureAdjusted} in the sense that $(\hvn, \hq, \hqO')$ solves \eqref{eq:WeakForm:TwoPressure:Trafo:2}--\eqref{eq:WeakForm:TwoPressure:Trafo:4} with \eqref{eq:WeakForm:TwoPressure:Trafo:1:PressureAdjusted} if and only if ($\vn, q, \qO)$ solves \eqref{eq:WeakForm:TwoPressure}, where 
\begin{align*}
&\hvn\txy = \An\txy \vn(t,x,y,\psin\txy) \, ,
\qquad 
\hq\tx = q\tx \, , \\
&\qquad \hqO' \txy = \qO(t,x,\psin\txy) \, .
\end{align*}
The latter identity is equivalent to 
\begin{align*}
\hqO\txy - (\psin\txy-y) \cdot \nabla_x \hat{q}\tx =\hqO'\txy = \qO(t,x,\psin\txy) 
\end{align*}
and, hence, $\hqO\txy = \qO(t,x,\psin\txy) + (\psin\txy-y) \cdot \nabla_x \hat{q}\tx$, which corresponds with back-transformation rules of the correctors of the gradients derived in \cite{AA23}.

Having the equivalence of the $\e$-scaled problems and the two-pressure Stokes problem, we can transfer also the two-scale convergence of $\hwe$, $\hve$ and $\hqe$ to $\we$, $\ve$ and $\qe$, respectively.
We note that we have to extend the functions in order to define the (two-scale) convergence.
We extend $\we$ and $\ve$ by zero as we did already for the transformed functions $\hwe, \hve$, i.e.~we denote by $\widetilde{\we}$ and $\widetilde{\hve}$ their extension by $0$ to $\Omega$. Since the transformation mapping $\psie$ is defined on all of $\Omega$, we have $\widetilde{\hwe}\tx= \widetilde{\we}(t,\psie\tx)$ and $\widetilde{\hve}\tx= \widetilde{\ve}(t,\psie\tx)$ for a.e.~$\tx \in (0,T) \times \Omega$. For the extension $\hQe$ of the pressure $\hqe$, we have used the extension by its cell-wise mean value. Consequently, in $\Oes$ this extension depends on the chosen $\psie$. In order to formulate the convergence result independently of $\psie$, we define analogously the extension $\Qe$ of $\qe$ to $\Omega$ by
\begin{align}\label{eq:Extension:Qe}
\Qe\tx \coloneqq 
\begin{cases}
	\qe\tx &\textrm{for } t \in (0,T), \, x\in \Oe(t) \, ,
	\\
	\frac{1}{|\Oe \cap \e (k + Y)|}\int\limits_{\e(k + \Yp)}\hqe\tx &\textrm{for } 
	t \in (0,T), \, x \in \Oes \cap \e (k + Y) \, \textrm{for } k \in \Ke \, .
\end{cases}
\end{align}
We note that ${\hQe}\tx= {\Qe}(t,\psie\tx)$ holds for $\tx \in (0,T) \times \Oe$ but not necessarily for $\tx \in (0,T) \times \Oes$. Nevertheless, this equivalence suffices in order to translate the convergence of $\widetilde{\hQe}$ to the convergence of $\widetilde{\Qe}$.

\begin{thm}
Let $(\we, \qe)$ be the solution of \eqref{eq:Weak:Stokes:Eps} and $\widetilde{\we}$ the extension of $\we$ by zero, $\Qe$ the extension of $\qe$ defined by \eqref{eq:Extension:Qe} and $\widetilde{\vn}$ the extension of $\vn$ by zero to $Y$. Then,
\begin{align}
\widetilde{\we} \tsw{2,2} \widetilde{v_0} \qquad \Qe \tsw{2,2} q \, ,
\end{align}
and, thus, in particular,
\begin{align}
\we \rightharpoonup v = \intYptx \vn \dy \quad \textrm{ in } L^2((0,T) \times \Omega) \, , \qquad \Qe \rightharpoonup q \textrm{ in } L^2((0,T) \times \Omega) \, ,
\end{align}
where
$(\vn, q, \qO)$ is the solution of \eqref{eq:WeakForm:TwoPressure}.
\end{thm}
\begin{proof}
Lemma \ref{lem:TwoScaleEquiv}, gives the two-scale convergence of $\we$ and $\hQe(t,\psie\tx)$ to $(\vn, q)$ with $\vn\txy = \hvn(t,x,\psinm\txy)$ and $q = \hat{q}$, where $(\hvn, \hq)$ are the two-scale limits of $\hwe$ and $\Qe$. Arguing similarly to \cite{JDE24}, one can deduce from the weak convergence of $\hQe(t,\psie\tx)$ to $q$ the weak convergence of $\Qe$ to $q$.
Since $(\hvn, \hq)$ is the first part of the solution of \eqref{eq:WeakForm:TwoPressure:Trafo:2}--\eqref{eq:WeakForm:TwoPressure:Trafo:4} with \eqref{eq:WeakForm:TwoPressure:Trafo:1:PressureAdjusted}, $(\vn, q)$ is the first part of the solution of \eqref{eq:WeakForm:TwoPressure}.
\end{proof}

In the case of a no-slip boundary condition at $\Ge(t)$, i.e.~in the case that $\vGe = \partial_t \psie(t,\psiem\tx)$ one can again simplify the right-hand side of the macroscopic divergence condition as in Lemma \ref{lem:NoSlipSimplification} to
\begin{align*}
\div_x \Big(\intYptx v_0\txy \dy \Big) &= -\frac{\dd}{\dt} \Theta\tx &&\textrm{in } (0,T) \times \Omega \,.
\end{align*} 

In order to separate the micro- and macroscopic variable in \eqref{eq:Strong:two-pressure}, we define 
\begin{align}
	v\tx \coloneqq \intYptx v_0 \txy \dy \, 
\end{align}
and note that \eqref{eq:hvn=vn} shows $\hv = v$.

Separating the micro- and macroscopic variable in \eqref{eq:Strong:two-pressure} for the case of the no-slip boundary condition leads to the Darcy law \eqref{eq:Strong:DarcyMemory}, with the permeability tensor \eqref{eq:def:K}, the initial value \eqref{eq:initialValue:Darcy} and the cell problems \eqref{eq:Strong:CellProblem:d} and \eqref{eq:Strong:CellProblem:init}.
For the general case, one gets the same result but with \eqref{eq:Strong:DarcyMemory:GenDirich} instead of \eqref{eq:Strong:DarcyMemory:2}
\begin{align}\label{eq:Strong:DarcyMemory:GenDirich}
\div_x v\tx = -\intYptx \div_y (\vGn\txy) \dy = - \int\limits_{\Gtx}\vGn\txy \cdot n \dy \, .
\end{align}

\appendix
\section{Two-scale convergence}
For the limit process, we use the notion of two-scale convergence, which was introduced in \cite{Ngu89, All92}, see also \cite{LNW02}.

Since we consider a time-dependent problem, we use the following with time parametrised version of two-scale convergence. In the following let $Y = (0,1)^d$.
\begin{defi}[Two-scale convergence]\label{def:two-scale}
Let $p\in [1,\infty)$, $q\in (1,\infty)$ and $\tfrac{1}{p} + \tfrac{1}{p'}=1$, $\tfrac{1}{q} + \tfrac{1}{q'}=1$. A sequence $\ue \in L^p(0,T;L^q(\Omega))$ two-scale converges to $\un \in L^p(0,T;L^q(\Omega \times Y))$ if
\begin{align*}
\lim\limits_{\e \to 0} \int\limits_{(0,T)} \intO \ue\tx \varphi \left(t, x, \frac{x}{\e}\right) \dxt = \intTOY \un\txy \varphi\txy \dyxt
\end{align*}
for all $\varphi \in L^{p'}(0,T;L^{q'}(\Omega;C_\#(Y)))$.
In this case, we write $\ue \tsw{p,q} \un$ or \linebreak $\ue\tx \tsw{p,q} \un\txy$ if we want to emphasize the dependence on the variables.

For $p, q \in (1,\infty)$, we say that $\ue$ strongly two-scale converges to $\un \in L^p(0,T;L^q(\Omega \times Y))$ if and only if $\ue \tsw{p,q} \un$ and $\lim\limits_{\e \to 0} \|\ue\|_{ L^p(0,T;L^q(\Omega))} =
\|\un\|_{L^p(0,T;L^q(\Omega \times Y))}$. In this case, we write $\ue \tss{p,q} \un$, or $\ue\tx \tss{p,q} \un\txy$ if we want to emphasize the dependence on the variables.
\end{defi}
Moreover, we have the following well-known compactness result.
\begin{lemma}[Two-scale compactness]
Let $p,q \in (1,\infty)$ and $\ue$ a bounded sequence in $ L^p(0,T;L^q(\Omega))$. Then, there exist a subsequence $\ue$ and $\un \in L^p(0,T;L^q(\Omega \times Y))$ such that for this subsequence $\ue \tsw{p,q} \un$. Moreover, if also $\e \nabla_x \hue$ is bounded in $L^p(0,T;L^q(\Omega)$, there exist a subsequence $\ue$ and $\un \in L^p(0,T;L^q(\Omega;W^{1,q}_\#(Y)))$ such that $\ue \tsw{p,q} \un$ and $\e \nabla \ue \tsw{p,q} \nabla_y \hun$.
\end{lemma}
Due to the transformation of the equations in the reference coordinates, we obtain coefficients which strongly two-scale converge. For those, the following two product results becomes useful. They can be derived for instance using the unfolding operator.
\begin{lemma}
Let $p, q,q_1, p_2, q_2 , p,q \in [1, \infty)$ with $\tfrac{1}{p_1} + \tfrac{1}{p_2} = \tfrac{1}{p}$, $\tfrac{1}{q_1} + \tfrac{1}{q_2} = \tfrac{1}{q}$ and $\ue$ be a sequence in $L^{p_1}(0,T;L^{q_1}(\Omega))$ and $\un \in L^{p_1}(0,T;L^{q_1}(\Omega \times Y))$ such that $\ue \tss{p_1, q_1} \un$. 
Let $\ve$ be a sequence in $L^{p_2}(0,T;L^{q_2}(\Omega))$ and $\vn \in L^{p_2}(0,T;L^{q_2}(\Omega \times Y))$ such that $\ve \tsw{p_2,q_2} \vn$. Then, $\ue\ve \tsw{p,q} \un\vn$.

Moreover, if also $p,q \in (1, \infty)$ and $\ve \tss{p_2,q_2} \vn$ one has $\ve \tss{p,q} \vn$\,.
\end{lemma}
In the case that the sequence $\ue$ is also essentially bounded, one can preserve the integrability.
\begin{lemma}
Let $p,q \in (1, \infty)$. Let $\ue$ be a bounded sequence in $L^1((0,T)\times \Omega)\cap L^\infty((0,T)\times \Omega)$ and $\un \in L^1((0,T)\times \Omega)\cap L^\infty((0,T)\times \Omega)$ such that $\ue \tss{2, 2} \un$. 
Let $\ve$ be a sequence in $L^{p}(0,T;L^{q}(\Omega))$ and $\vn \in L^{p}(0,T;L^{q}(\Omega \times Y))$ such that $\ve \tsw{p,q} \vn$ (resp.~$\ve \tss{p,q} \vn$). Then, $\ue\ve \tsw{p,q} \un\vn$ (resp.~$\ue\ve \tss{p,q} \un\vn$).
\end{lemma}

\subsection*{Transformation and two-scale convergence}
Following \cite{AA23}, we obtain the following result on two-scale convergence in the context of the microscopic coordinate transformation. 
\begin{lemma}[Bounds for the Jacobians]\label{lem:Est:Psie}
Let $\psie$ satisfy Assumption \ref{ass:psie}\ref{item:R1}--\ref{item:R2}, Assumption \ref{ass:psie}\ref{item:B1} for $l = 1$.
Then, there exists a constant $C$ such that
\begin{align*}
\| \Pe\|_{L^\infty((0,T) \times \Omega)} + \| \Pem\|_{L^\infty((0,T) \times \Omega)} + \| \Je\|_{L^\infty((0,T) \times \Omega)}&\leq C \, ,
\\
 \| \Jem\|_{L^\infty((0,T) \times \Omega)} + \| \Ae\|_{L^\infty((0,T) \times \Omega)} + \| \Aem\|_{L^\infty((0,T) \times \Omega)} &\leq C \, ,
\\
\|\partial_t \Pe\|_{L^\infty((0,T) \times \Omega)} + \|\partial_t \Pem\|_{L^\infty((0,T) \times \Omega)} + \|\partial_t \Je\|_{L^\infty((0,T) \times \Omega)} &\leq C \, ,
\\
\|\partial_t \Jem\|_{L^\infty((0,T) \times \Omega)} +\|\partial_t \Ae\|_{L^\infty((0,T) \times \Omega)} + \|\partial_t \Aem\|_{L^\infty((0,T) \times \Omega)} &\leq C \, .
\end{align*}
Assume that additionally Assumption~\ref{ass:psie}\ref{item:B1} is satisfied for $l = 2$.
Then, there exists a constant $C$ such that
\begin{align*}
\e \|\partial_x \Pe\|_{L^\infty((0,T) \times \Omega)} + \e \|\partial_x\Pem\|_{L^\infty((0,T) \times \Omega)} + \e\|\partial_x\Je\|_{L^\infty((0,T) \times \Omega)} &\leq C \, ,
\\
\e \|\partial_x \Jem\|_{L^\infty((0,T) \times \Omega)} + \e\|\partial_x\Ae\|_{L^\infty((0,T) \times \Omega)} + \e\|\partial_x \Aem\|_{L^\infty((0,T) \times \Omega)} &\leq C \, ,
\\
\e \|\partial_t \partial_x \Pe\|_{L^\infty((0,T) \times \Omega)} + \e \|\partial_t \partial_x\Pem\|_{L^\infty((0,T) \times \Omega)} + \e\|\partial_t \partial_x\Je\|_{L^\infty((0,T) \times \Omega)} &\leq C \, ,
\\
\|\e \partial_x \partial_t \Jem\|_{L^\infty((0,T) \times \Omega)} + \e\|\partial_x \partial_t \Ae\|_{L^\infty((0,T) \times \Omega)} + \e\|\partial_x \partial_t \Aem\|_{L^\infty((0,T) \times \Omega)} &\leq C \, .
\end{align*}
\end{lemma}
\begin{proof}
Lemma \ref{lem:Est:Psie} can be shown by rewriting all quantities in terms of polynomials in $\partial_x \psie$, $\partial_x \partial_t \psie$, $\e \partial_x \partial_x \psie$, $\e \partial_x \partial_x \partial_t \psie$ and $\Jem$ which are bounded by the assumption. 
We refer to \cite{AA23}, \cite{NA23}, \cite{JDE24} for a more detailed proof.
\end{proof}

\begin{lemma}[Bounds under transformations]\label{lem:BoundsUnderTransformation}
Let $\psie$ satisfy Assumption \ref{ass:psie}\ref{item:R1}--\ref{item:R2}, Assumption \ref{ass:psie}\ref{item:B1} for $l = 1$.
Let $p,q \in [1,\infty)$ and $\ue$ be a sequence of functions and $\hue\tx \coloneqq \ue(t, \psie\tx)$. Then,
\begin{align*}
&\ue \in L^p(0,T;L^q(\Oe(t))) &&\textrm{ if and only if } &&\hue \in L^p(0,T;L^q(\Oe)) \, ,
\\
&\ue \in L^p(0,T;W^{1,q}(\Oe(t))) &&\textrm{ if and only if } &&\hue \in L^p(0,T;W^{1,q}(\Oe)) \, ,
\\
&\ue \in H^1(\Omega_\e^T) &&\textrm{ if and only if } &&\hue \in H^1((0,T) \times \Oe) \, .
\end{align*}
Moreover, in all three cases, $\ue$ is uniformly bounded if and only if $\hue$ is uniformly bounded.
\end{lemma}
The space $L^p(0,T;L^q(\Oe(t)))$ can be understood as the measurable functions on $\Oe^T$ such that $\| \|u(\cdot)\|_{L^q(\Oe(t))}\|_{L^p(0,T)}$ is finite. The space has to be understood as the subset of $L^p(0,T;L^q(\Oe(t)))$ such that $u(t) \in W^{1,q}(\Oe(t))$ for a.e.~$t\in \Omega$ and $\partial_x u \in L^p(0,T;L^q(\Oe(t)))$.

\begin{proof}[Proof of \ref{lem:BoundsUnderTransformation}]
Lemma \ref{lem:BoundsUnderTransformation} follows directly from coordinate transformation and the uniform essential bounds for the Jacobians of $\psie$.
\end{proof}

For the Jacobians of the transformations, we obtain the following asymptotic behaviour:
\begin{lemma}[Two-scale convergence of the Jacobians]\label{lem:TwoScaleJacobians}
	Let $\psie$ satisfy Assumption \ref{ass:psie}\ref{item:R1}--\ref{item:R2}, Assumption \ref{ass:psie}\ref{item:B1} for $l = 1$ and Assumption \ref{ass:psie}\ref{item:A1}--\ref{item:A2}, \ref{ass:psie}\ref{item:A4}--\ref{item:A5}
	Then, for every $p \in (1,\infty)$,
	\begin{align*}
		&\Pe \tss{p,p} \Pn \,, \quad 
		&&\Pem \tss{p,p} \Pnm \,, \quad 
		&&\Je \tss{p,p} \Jn \,, \quad 
		&&\Jem \tss{p,p} \Jnm \,, \quad 
		\\
		&\Ae \tss{p,p} \An \,, \quad 
		&&\Aem \tss{p,p} \Anm \,, \quad
		&&\partial_t\Pe \tss{p,p} \partial_t \Pn \,, \quad 
		&&\partial_t\Pem \tss{p,p} \partial_t \Pnm \,, \quad 
		\\
		&\partial_t\Je \tss{p,p} \partial_t \Jn \,, \quad 
		&&\partial_t\Jem \tss{p,p} \partial_t \Jnm \,, \quad 
		&&\partial_t\Ae \tss{p,p} \partial_t \An \,, \quad 
		&&\partial_t\Aem \tss{p,p} \partial_t \Anm \,. \quad
	\end{align*}
	Assume that additionally Assumption~\ref{ass:psie}\ref{item:B1} is satisfied for $l = 2$ and \ref{ass:psie}\ref{item:A3},\ref{item:A6}. Then, one has additionally
	\begin{align*}
		&\e \partial_x\Pe \tss{p,p} \partial_y\Pn \,, \quad 
		&&\e \partial_x\Pem \tss{p,p} \partial_y\Pnm \,, \quad 
		&&\e \partial_x\Je \tss{p,p} \partial_y\Jn \,, \quad 
		\\
		&\e \partial_x\Jem \tss{p,p} \partial_y \Jnm \,, \quad 
		&&\e \partial_x\Ae \tss{p,p} \partial_y\An \,, \quad 
		&&\e \partial_x\Aem \tss{p,p} \partial_y\Anm \,, \quad
		\\
		&\e \partial_x\partial_t\Pe \tss{p,p} \partial_y\partial_t \Pn \,, \quad 
		&&\e \partial_x\partial_t\Pem \tss{p,p} \partial_y\partial_t \Pnm \,, \quad 
		&&\e \partial_x\partial_t\Je \tss{p,p} \partial_y\partial_t \Jn \,, \quad 
		\\
		&\e \partial_x\partial_t\Jem \tss{p,p} \partial_y \partial_t \Jnm \,, \quad 
		&&\e \partial_x\partial_t\Ae \tss{p,p} \partial_y\partial_t \An \,, \quad 
		&&\e \partial_x\partial_t\Aem \tss{p,p} \partial_y\partial_t \Anm \,. \quad
	\end{align*}
\end{lemma}
\begin{proof}
The first part was shown for the time-independent case in \cite{AA23}, the time-dependent case can be deduced by the same argumentation and is given also in \cite{NA23}. The second part becomes relevant for the Stokes problem and is presented for the time-independent case in \cite{JDE24}, while the time-dependent case can be deduced by the same argumentation.
\end{proof}
Moreover, we can translate the two-scale convergence between the transformed and untransformed setting as follows.
\begin{lemma}[Transformation and two-scale convergence]\label{lem:TwoScaleEquiv}
Let $\psie$ satisfy Assumption \ref{ass:psie}\ref{item:R1}--\ref{item:R2}, Assumption \ref{ass:psie}\ref{item:B1} for $l = 1$ and let $\psin$ satisfy Assumption \ref{ass:psie}\ref{item:L1}--\ref{item:L3} such that the convergence of Assumption \ref{ass:psie}\ref{item:A1}--\ref{item:A2} is satisfied.
Let $p,q \in (1,\infty)$ and $\ue \in L^p(0,T;L^q(\Omega))$ be a sequence of functions, $\hue\tx \coloneqq \ue(t, \psie\tx)$ and $\un \in L^p(0,T;L^q(\Omega \times Y))$. Then,
\begin{align*}
\ue &\tsw{p,q} \un &&\textrm{ if and only if } && \hue \tsw{p,q} \hun \, ,
\\
\ue &\tss{p,q} \un &&\textrm{ if and only if } && \hue \tss{p,q} \hun \, ,
\end{align*}
where $\hun\txy = \un(t,x,\psin\txy)$.

Moreover, if $\ue \in L^p(0,T;W^{1,q}(\Omega))$, $\un \in L^p(0,T;
W^{1,q}(\Omega))$, \linebreak $u_1 \in L^p(0,T;
L^q(\Omega;W^{1,q}_\#(Y)))$ one has
\begin{align*}
&\nabla \ue \tsw{p,q} \nabla_x \un + \nabla_y u_1 &&\textrm{ if and only if } && \nabla \hue \tsw{p,q} \hun \nabla_x \hun + \nabla_y \hat{u}_1 \, 
\end{align*}
for $\un =\hun$ and ${\hat{u}_1\txy = u_1(t,x,\psin\txy) + \nabla_x \un\tx \cdot (\psin\txy-y)}$\,.
If $\ue \in L^p(0,T;W^{1,q}(\Omega))$, $\un \in L^p(0,T;
L^q(\Omega;W^{1,q}_\#(Y)))$ one has
\begin{align*}
&\e \nabla \ue \tsw{p,q} \nabla_y \un &&\textrm{ if and only if } && \e \nabla \hue \tsw{p,q} \nabla_y \hun \, 
\end{align*}
for $\hun\txy = \un(t,x,\psin\txy)$\,.
\end{lemma}
\begin{proof}
For the proof of the time-independent case, see \cite{AA23}. The time-dependent case can be deduced by the same argumentation.
\end{proof}

\begin{rem}\label{rem:two-scale-trafo:FixedPointTime}
Lemma \ref{lem:Est:Psie}--Lemma~\ref{lem:TwoScaleEquiv} are formulated for the two-scale convergence with the time as parameter. Since Assumption \ref{ass:psie} provides also a uniform control of the time derivative of $\psie$, the two-scale convergence of $\psie$ and its spatial derivatives holds also for every fixed point in time, in particular for the initial time. Thus, one can also deduce Lemma \ref{lem:Est:Psie}--Lemma~\ref{lem:TwoScaleEquiv} for a fixed point in time, which becomes useful for the investigation of the initial values.
\end{rem}
\bibliography{InstatStokes}
\end{document}